\documentclass[a4paper,10pt]{amsart}
\usepackage[top=0.8in, bottom=0.8in, left=1.25in, right=1.25in]{geometry}
\title{Genus zero Gopakumar-Vafa type invariants for Calabi-Yau 4-folds}
\date{}
\author{Yalong Cao}
\address{Kavli Institute for the Physics and Mathematics of the Universe (WPI),The University of Tokyo Institutes for Advanced Study, The University of Tokyo, Kashiwa, Chiba 277-8583, Japan}
\address{Mathematical Institute, University of Oxford, Andrew Wiles Building, Radcliffe Observatory Quarter, Woodstock Road, Oxford, OX2 6GG, U.K.}
\email{yalong.cao@ipmu.jp; yalong.cao@maths.ox.ac.uk}
\author{Davesh Maulik}
\address{Massachusetts Institute of Technology, Departement of Mathematics, 77
Massachusetts Avenue Cambridge, MA 02139, US.}
\email{maulik@mit.edu}
\author{Yukinobu Toda}
\address{Kavli Institute for the Physics and Mathematics of the Universe (WPI),The University of Tokyo Institutes for Advanced Study, The University of Tokyo, Kashiwa, Chiba 277-8583, Japan}
\email{yukinobu.toda@ipmu.jp}

\usepackage{fancybox}
\usepackage{amsmath}
\usepackage{amssymb}
\usepackage{amsthm}
\usepackage{mathrsfs}
\usepackage{epstopdf}
\usepackage[all]{xy}
\usepackage{epstopdf}
\usepackage{array}
\usepackage{amscd}

\usepackage{amsfonts}
\usepackage{amssymb}
\usepackage{tikz}

\DeclareFontFamily{U}{rsfs}{%
\skewchar\font127}
\DeclareFontShape{U}{rsfs}{m}{n}{%
<-6>rsfs5<6-8.5>rsfs7<8.5->rsfs10}{}
\DeclareSymbolFont{rsfs}{U}{rsfs}{m}{n}
\DeclareSymbolFontAlphabet
{\mathrsfs}{rsfs}
\DeclareRobustCommand*\rsfs{%
\@fontswitch\relax\mathrsfs}

\theoremstyle{plain}
\newtheorem{thm}{Theorem}[section]
\newtheorem{prop}[thm]{Proposition}
\newtheorem{lem}[thm]{Lemma}

\newtheorem{defi}[thm]{Definition}
\newtheorem{rmk}[thm]{Remark}
\newtheorem{cor}[thm]{Corollary}

\newtheorem{claim}[thm]{Claim}

\newtheorem{prop-defi}[thm]{Proposition-Definition}
\newtheorem{thm-defi}[thm]{Theorem-Definition}
\newtheorem{lem-defi}[thm]{Lemma-Definition}

\newtheorem{conj}[thm]{Conjecture}

\newdimen\argwidth
\def\db[#1\db]{
 \setbox0=\hbox{$#1$}\argwidth=\wd0
 \setbox0=\hbox{$\left[\box0\right]$}
  \advance\argwidth by -\wd0
 \left[\kern.3\argwidth\box0 \kern.3\argwidth\right]}

\newcommand{\aA}{\mathcal{A}}

\newcommand{\cC}{\mathcal{C}}

\newcommand{\eE}{\mathcal{E}}
\newcommand{\fF}{\mathcal{F}}
\newcommand{\gG}{\mathcal{G}}
\newcommand{\hH}{\mathcal{H}}
\newcommand{\iI}{\mathcal{I}}

\newcommand{\lL}{\mathcal{L}}

\newcommand{\nN}{\mathcal{N}}
\newcommand{\oO}{\mathcal{O}}

\newcommand{\uU}{\mathcal{U}}
\newcommand{\vV}{\mathcal{V}}

\newcommand{\zZ}{\mathcal{Z}}

\newcommand{\Supp}{\mathop{\rm Supp}\nolimits}
\newcommand{\Hom}{\mathop{\rm Hom}\nolimits}

\newcommand{\dR}{\mathbf{R}}
\newcommand{\dL}{\mathbf{L}}

\newcommand{\Hilb}{\mathop{\rm Hilb}\nolimits}

\newcommand{\Pic}{\mathop{\rm Pic}\nolimits}

\newcommand{\id}{\textrm{id}}

\newcommand{\ch}{\mathop{\rm ch}\nolimits}

\newcommand{\Ext}{\mathop{\rm Ext}\nolimits}
\newcommand{\Spec}{\mathop{\rm Spec}\nolimits}

\newcommand{\Coh}{\mathop{\rm Coh}\nolimits}

\newcommand{\cneq}{\mathrel{\raise.095ex\hbox{:}\mkern-4.2mu=}}
\newcommand{\eqcn}{\mathrel{=\mkern-4.5mu\raise.095ex\hbox{:}}}

\newcommand{\Cok}{\mathop{\rm Cok}\nolimits}

\newcommand{\DT}{\mathop{\rm DT}\nolimits}

\newcommand{\Ker}{\mathop{\rm Ker}\nolimits}

\newcommand{\RHom}{\mathop{\dR\mathrm{Hom}}\nolimits}

\makeatletter
 
  \@addtoreset{equation}{section}
\makeatother

\begin{document}
\maketitle
\begin{abstract}
In analogy with the Gopakumar-Vafa conjecture on CY 3-folds, Klemm and Pandharipande defined GV type invariants on Calabi-Yau 4-folds using Gromov-Witten theory
and conjectured their integrality. In this paper, we propose
 a sheaf-theoretic interpretation of their genus zero
 invariants using Donaldson-Thomas theory on CY 4-folds.
More specifically, we conjecture genus zero GV type invariants are $\mathrm{DT_{4}}$ invariants for one-dimensional stable sheaves
on CY 4-folds.
Some examples are computed for both compact and non-compact CY 4-folds to 
support our conjectures. We also propose an equivariant version of the conjectures for local curves and verify
them in certain cases.
\end{abstract}

\tableofcontents

\section{Introduction}
\subsection{Background}
Gromov-Witten invariants \cite{BF, LT} are rational numbers which virtually count stable maps from complex curves to algebraic varieties (or symplectic manifolds).
Because of multiple-cover contributions, they are in general not integers and hence are not honest enumerative invariants. On a Calabi-Yau 3-fold $Y$, motivated by
string duality,  Gopakumar-Vafa \cite{GV} conjectured the existence of integral invariants $n_{g,\beta}$ ($g\geqslant0$, $\beta\in H_2(Y)$) which determine Gromov-Witten invariants $\mathrm{GW}_{g,\beta}$ by the identity
\begin{equation} \sum_{\beta>0,g\geqslant0}\mathrm{GW}_{g,\beta} \lambda^{2g-2}t^{\beta}= \sum_{\beta>0, g\geqslant0, k\geqslant1}\frac{n_{g,\beta}}{k}\Big(2\sin\left(\frac{k\lambda}{2}\right)  \Big)^{2g-2}t^{k\beta} .  \nonumber \end{equation}
In particular, when $g=0$, it recovers the Aspinwall-Morrison multiple cover formula
\begin{align*}
\mathrm{GW}_{0, \beta}=\sum_{k\geqslant 1, k|\beta} \frac{1}{k^3}
n_{0, \beta/k}.
\end{align*}
Moreover, the invariants $n_{g, \beta}$ should be interpreted in
a sheaf-theoretic way~\cite{GV, HST, Katz, KL3, MT}, for example when $g=0$, the invariant $n_{0, \beta}$ is conjectured to be the
Donaldson-Thomas invariant~\cite{Thomas} counting one-dimensional stable sheaves $E$ with $[E]=\beta$, $\chi(E)=1$ on $Y$.

As Gromov-Witten invariants can be defined for smooth varieties of any dimension, it is natural to ask for generalizations of GV type invariants in higher dimensions. In~\cite{KP}, Klemm-Pandharipande gave a definition of
GV type invariants on Calabi-Yau 4-folds via Gromov-Witten theory,
and conjectured that they are integers.

\subsection{GV type invariants on CY 4-folds}
Let $X$ be a smooth projective Calabi-Yau 4-fold \cite{Yau}.
Note that in this case,
Gromov-Witten invariants
vanish for genus $g\geqslant2$ by the
dimension reason, so one only needs to consider the genus 0 and 1 cases.

The genus 0 GW invariants on $X$ are defined using
insertions:
 for integral classes $\gamma_i \in H^{m_i}(X, \mathbb{Z}), \
1\leqslant i\leqslant n$,
one defines
\begin{equation}
\mathrm{GW}_{0, \beta}(\gamma_1, \ldots, \gamma_n)
:=\int_{[\overline{M}_{0, n}(X, \beta)]^{\rm{vir}}}
\prod_{i=1}^n \mathrm{ev}_i^{\ast}(\gamma_i),
\nonumber \end{equation}
where $\mathrm{ev}_i \colon \overline{M}_{0, n}(X, \beta)\to X$
is the $i$-th evaluation map.
The invariants
\begin{align}\label{intro:n}
n_{0, \beta}(\gamma_1, \ldots, \gamma_n) \in \mathbb{Q}
\end{align}
are defined in \cite{KP} by the identity
\begin{align*}
\sum_{\beta>0}\mathrm{GW}_{0, \beta}(\gamma_1, \ldots, \gamma_n)q^{\beta}=
\sum_{\beta>0}n_{0, \beta}(\gamma_1, \ldots, \gamma_n) \sum_{d=1}^{\infty}
d^{n-3}q^{d\beta}.
\end{align*}
\begin{conj}\emph{(\cite{KP})}
The invariants 
$n_{0, \beta}(\gamma_1, \ldots, \gamma_n)$ are integers.
\end{conj}
In~\cite{KP}, 
genus zero 
GW invariants on $X$ are computed directly in many examples, 
using either virtual localization or mirror symmetry techniques, in support of their conjectures.

\subsection{Our proposal}
The aim of this paper is to give a sheaf-theoretic interpretation of the above GV type invariants (\ref{intro:n}) in terms of
Donaldson-Thomas invariants for CY 4-folds
(called $\mathrm{DT}_4$-invariants)
introduced by Cao-Leung~\cite{CL}
and Borisov-Joyce~\cite{BJ}.
More specifically, we consider the moduli space $M_\beta$ of 1-dimensional stable sheaves with Chern character $(0,0,0,\beta,1)$.
By the results of~\cite{CL, BJ}, assuming the existence of a suitable orientation on this space,
there exists a $\DT_4$-virtual class
\begin{align}\label{intro:DT4vir}
[M_{\beta}]^{\rm{vir}} \in H_{2}(M_{\beta}, \mathbb{Z}). 
\end{align}
The virtual class (\ref{intro:DT4vir}) depends on the choice of an orientation on certain line bundle.  On each connected component
of $M_{\beta}$, there are two choices of orientations, which affect the corresponding contribution to the class
(\ref{intro:DT4vir}) by a sign (for each connected component). 

In order to define the invariants, we require 
some insertions. 
Let $\eE$ be a universal sheaf on $X \times M_{\beta}$. 
We define the map $\tau$ by 
\begin{align*}
\tau \colon H^{m}(X)\to H^{m-2}(M_{\beta}), \
\tau(\gamma)=\pi_{M\ast}(\pi_X^{\ast}\gamma \cup\ch_3(\eE) ),
\end{align*}
where $\pi_X$, $\pi_M$ are projections from $X \times M_{\beta}$
to corresponding factors, and $\ch_3(\eE)$ is the
Poincar\'e dual to the
fundamental cycle of the universal sheaf $\eE$.
For $\gamma_i \in H^{m_i}(X, \mathbb{Z}), \
1\leqslant i\leqslant n$, the $\mathrm{DT}_{4}$ invariant
is defined by 
\begin{align*}\mathrm{DT_{4}}(\beta \mid \gamma_1,\ldots,\gamma_n):=\int_{[M_{\beta}]^{\rm{vir}}} \prod_{i=1}^{n}\tau(\gamma_i). 
 \end{align*}

\begin{conj}\emph{(Conjecture \ref{conj:GW/GV})} \label{intro:conj:GW/GV}
For a suitable choice of orientation, we
have the identity
\begin{align*}
n_{0,\beta}(\gamma_1, \ldots, \gamma_n)=
\mathrm{DT}_{4}(\beta \mid \gamma_1, \ldots, \gamma_n).
\end{align*}
In particular, we have the multiple cover formula
\begin{align*}
\mathrm{GW}_{0, \beta}(\gamma)=
\sum_{k|\beta}\frac{1}{k^{2}}\cdot\mathrm{DT}_{4}(\beta/k \mid \gamma).
\end{align*}
\end{conj}
In the current formulation, we do not specify how to choose the orientation in order to match invariants.  While a priori there are many choices
of orientations, gauge theory arguments (explained later) suggest there are deformation invariant orientations.  However, specifying the choice among them requires further investigation.  \\

Our proposal is based on the
heuristic argument in Subsection~\ref{subsec:geometric},
 where we prove
Conjecture~\ref{intro:conj:GW/GV}
assuming the CY 4-fold $X$ to be 'ideal', i.e. curves in $X$ are smooth of expected dimensions.
Apart from that, we verify our conjecture in examples as follows.
\subsection{Verifications of the conjecture I: compact examples}
We first prove Conjecture~\ref{intro:conj:GW/GV}
in some examples of compact CY 4-folds. 

${}$  \\
\textbf{Elliptic fibrations}. 
We consider a projective CY 4-fold $X$ which
admits an elliptic fibration 
\begin{align*}
\pi \colon X\to\mathbb{P}^{3}
\end{align*}
 over $\mathbb{P}^{3}$, given by a Weierstrass model. 
Let $f$ be a general fiber of $\pi$. Then we have  
\begin{prop}\emph{(Proposition \ref{prop on elliptic fib})} 
For multiple fiber classes $\beta=r[f]$, $r\geqslant1$, 
Conjecture~\ref{intro:conj:GW/GV} is true.  
\end{prop}
In this case, we can directly compute the $\DT_4$ invariants and 
check the compatibility with the computation 
of GW invariants in~\cite{KP}. 

${}$  \\
\textbf{CY 3-fold fibrations}. 
We consider a projective CY 4-fold $X$ 
which admits a CY 3-fold fibration 
\begin{align*}
\pi \colon X\to C
\end{align*}
 over a curve $C$.
In this case, 
 we conjecture a $\mathrm{DT}_{4}/\mathrm{DT}_{3}$ correspondence (Conjecture \ref{conj on CY4/CY3}), which roughly says that $\mathrm{DT}_{4}$ invariants for one-dimensional stable sheaves supported on general fibers of $\pi$ equal $\mathrm{DT}_{3}$ invariants for one-dimensional stable sheaves on those general fibers (CY 3-folds).
A special case is when
\begin{align}\label{intro:product}
X=Y\times E
\end{align}
 where $Y$ is a CY 3-fold, 
$E$ is an elliptic curve and 
$\pi$ is the projection to $E$. 
In this situation, we verify the conjecture:
\begin{prop}\emph{(Corollary \ref{g=0 product of elliptic curve and CY3})}
Suppose $X$ is given by (\ref{intro:product}).
Then for any $\beta\in H_{2}(Y)\subseteq H_{2}(X)$ and divisor $H\subseteq X$, we have  
\begin{equation}\label{DT4/DT3 invs}\mathrm{DT}_{4}(\beta\mid H\cdot Y)=\mathrm{DT}_{3}(\beta)\cdot(H\cdot\beta), \end{equation}
for certain choice of orientation in defining the LHS. Here $\mathrm{DT}_{3}(\beta)=\deg[M_{Y,\beta}]^{\rm{vir}}$ is the $\mathrm{DT_{3}}$ invariant~\cite{Thomas} for one dimensional stables sheaves with 
Chern character given by $(0,0,\beta,1)$. 
\end{prop}
This would then imply that our Conjecture ~\ref{intro:conj:GW/GV} is consistent with the genus zero GV/DT conjecture on CY 3-folds \cite{Katz}. In particular, by combining the GW/DT/PT correspondence, geometric vanishing and wall-crossing formula on CY 3-folds (see Corollary \ref{cor:CI}), we obtain
the following result: 
\begin{thm}\emph{(Theorem \ref{DT4/DT3 primitive})}\label{intro:thm:DT4/DT3}
Let $Y$ be a complete intersection CY 3-fold in a product of projective 
spaces and let $X=Y \times E$ for an elliptic curve $E$. 
Then for any primitive curve class $\beta\in H_{2}(Y)\subseteq H_{2}(X)$ and divisor $H\subseteq X$, we have
\begin{equation}\mathrm{GW}_{0, \beta}(H \cdot Y)=\DT_4(\beta \mid H \cdot Y), \nonumber \end{equation}
for certain choice of orientation in defining the RHS, i.e. Conjecture ~\ref{intro:conj:GW/GV} holds in this case.  
\end{thm}

${}$  \\
\textbf{Hyperk\"ahler 4-folds}. When the CY 4-fold $X$ has a holomorphic symplectic form, i.e. $X$ is a hyperk\"ahler 4-fold, GW invariants 
vanish, and so do the GV type invariants. To verify 
Conjecture~\ref{intro:conj:GW/GV}, we are left to prove the vanishing
of $\mathrm{DT_4}$ invariants. In \cite{CL}, such vanishing is shown for torsion-free sheaves by considering the trace map, but this argument does not apply to the case of torsion sheaves. Instead, we construct a cosection map from the (trace-free) obstruction sheaf of moduli spaces of stable sheaves to a trivial bundle which is compatible with Serre duality (Proposition \ref{surj cosection}).
We expect the following vanishing result then follows. 
\begin{claim}\emph{(Claim \ref{vanishing for hk4})}
Let $X$ be a projective hyperk\"{a}hler 4-fold and $M$ be a proper moduli scheme of simple perfect complexes $F$'s with 
$\ch_{4}(F)\neq 0$ or $\ch_{3}(F)\neq 0$. Then the virtual class  
$[M]^{\rm{vir}}\in H_*(M)$ vanishes.
\end{claim}
At the moment, we are lack of Kiem-Li type theory of cosection localization for D-manifolds in the sense of Joyce or Kuranishi space structures in
the sense of Fukaya-Oh-Ohta-Ono. We believe that when such a theory is established, our claim should follow automatically. 
Nevertheless, we have the following evidence for the claim.

1. At least when $M_\beta$ is smooth, Proposition \ref{surj cosection} gives the vanishing of virtual class.

2. If there is a complex analytic version of $(-2)$-shifted symplectic geometry \cite{PTVV} and the corresponding construction of virtual classes \cite{BJ},
one could prove the vanishing result as in $\mathrm{GW}$ theory, i.e. taking a generic complex structure in the $\mathbb{S}^{2}$-twistor family 
of the hyperk\"ahler 4-fold which does not support coherent sheaves and then vanishing of virtual classes follows from their deformation invariance. \\

By taking away trivial obstruction factors, one can define reduced virtual classes and the corresponding invariants. 
We explicitly do this in an example and prove a version of Conjectures~\ref{intro:conj:GW/GV} for reduced invariants (Proposition \ref{g=0 reduced inv conj}). 

\subsection{Verifications of the conjecture II: local surfaces}
For a smooth projective surface $S$, we 
consider the 
non-compact CY 4-fold 
\begin{align*}
X=\mathrm{Tot}_{S}(L_1\oplus L_2)
\end{align*}
 where $L_1$, $L_2$ are line bundles on $S$ 
satisfying $L_1\otimes L_2\cong K_S$. 
We can also investigate an analogue of 
Conjecture~\ref{intro:conj:GW/GV} previously formulated on projective CY 4-folds. In particular, when $L^{-1}_1$, $L^{-1}_2$ are ample,
the
moduli space $M_{X, \beta}$ of one dimensional stable sheaves on $X$ is
 compact (Proposition \ref{vir:loc neg}).
So $\mathrm{DT_4}$ invariants are well-defined 
and we can study Conjecture~\ref{intro:conj:GW/GV} in this case. 
In Subsection~\ref{subsec:locP2}, we
 check 
this for low degree curves when $S=\mathbb{P}^2$ and
\begin{align*}
X=\mathrm{Tot}_{\mathbb{P}^{2}}(\mathcal{O}_{\mathbb{P}^{2}}(-1)\oplus\mathcal{O}_{\mathbb{P}^{2}}(-2)).
\end{align*}

In general, the moduli space
$M_{X, \beta}$ is non-compact. On the other hand, 
there is a $\mathbb{C}^{\ast}$-action on 
$X$ fiberwise over $S$, 
which preserves the CY 4-form on $X$, and such that
the $\mathbb{C}^{\ast}$-fixed 
locus of $M_{X, \beta}$ is compact. 
In this case, $\DT_4$ invariants may be defined via equivariant residue, yielding
rational functions of the equivariant parameters.
Namely we define
\begin{align}\label{intro:localization}
[M_{X,\beta}]^{\rm{vir}}:=
[M_{X,\beta}^{\mathbb{C}^{\ast}}]^{\rm{vir}}
\cdot  e( \dR \hH om_{\pi_M}(\eE, \eE)^{\rm{mov}})^{1/2}
\in H_{\ast}(M_{X,\beta}^{\mathbb{C}^{\ast}})[t^{\pm 1}],
\end{align}
as in \cite[Section 8]{CL}. Here 
\begin{align*}
[M_{X,\beta}^{\mathbb{C}^{\ast}}]^{\rm{vir}}
\in H_{\ast}(M_{X,\beta}^{\mathbb{C}^{\ast}})
\end{align*} 
should be the $\mathrm{DT_{4}}$ virtual class of the $\mathbb{C}^{\ast}$-fixed locus, $\eE$ is the universal sheaf 
on $X \times M_{X, \beta}$,  
$\pi_M \colon X\times M_{X,\beta}\to M_{X,\beta}$ is the projection
and $t$ is the equivariant parameter for the $\mathbb{C}^{\ast}$-action. 
Since the localization formula for $\DT_4$-virtual class is not 
yet available, the definition of (\ref{intro:localization}) is only 
heuristic at this moment. 

Note that the moduli space $M_{S,\beta}$ of one-dimensional stable sheaves on $S$ is a union of connected components of $M_{X,\beta}^{\mathbb{C}^{\ast}}$ and we can determine its contribution to (\ref{intro:localization}) (Proposition \ref{equiv vir cycle surface cpn}). When the surface component is the only $\mathbb{C}^{\ast}$-fixed locus, i.e. $M_{X,\beta}^{\mathbb{C}^{\ast}}=M_{S,\beta}$,
we can rigorously define (\ref{intro:localization}) 
and the following residue $\DT_4$ invariant
$\DT_4^{\rm{res}}(\beta) \in \mathbb{Z}$ by taking the residue of 
(\ref{intro:localization})
at $t=0$:
\begin{align*}
\DT_4^{\rm{res}}(\beta):=\mathrm{Res}_{t=0}
\int_{[M_{X,\beta}^{\mathbb{C}^{\ast}}]^{\rm{vir}}}
e( \dR \hH om_{\pi_M}(\eE, \eE)^{\rm{mov}})^{1/2}. 
\end{align*}
Then we 
propose an analogue of Conjecture~\ref{intro:conj:GW/GV} 
(see Conjecture \ref{conj:red}) for residue invariants 
as follows
\begin{equation}\label{intro: residue invs conj}
\mathrm{GW}_{0, \beta}^{\rm{res}}=\sum_{k|\beta}\frac{1}{k^3} \DT_4^{\rm{res}}(\beta/k).
\end{equation}
Here $\mathrm{GW}_{0, \beta}^{\rm{res}}\in\mathbb{Q}$
is the corresponding residue GW invariants. Note that the power of $1/k$ becomes three (instead of two in (\ref{intro:conj:GW/GV})) as there is no insertion here. The above conjecture is verified in the following:  
\begin{thm}\emph{(Theorem \ref{toric del-Pezzo}, \ref{rational elliptic})}
Let $X=\mathrm{Tot}_S(\oO_S \oplus K_S)$. Then $M_{X,\beta}^{\mathbb{C}^{\ast}}=M_{S,\beta}$ and (\ref{intro: residue invs conj}) is true if 
\begin{enumerate}
\item $S$ is a smooth toric del-Pezzo surface and $\beta\in H_2(X)$ is any curve class; 
\item $S$ is a rational elliptic surface, and $\beta=\beta_n$ are primitive classes defined in (\ref{beta:n}).
\end{enumerate}
\end{thm}
Similarly to Theorem~\ref{intro:thm:DT4/DT3}, 
the above result is reduced to the 
similar result for the
non-compact CY 3-fold $Y=\mathrm{Tot}_S(K_S)$. 

\subsection{Verifications of the conjecture III: local curves}
Let $C$ be a smooth projective curve with genus $g(C)$. 
We consider a CY 4-fold $X$ given by 
\begin{align*}
X=\mathrm{Tot}_{C}(L_1\oplus L_2\oplus L_3),
\end{align*}
 where 
$L_1, L_2, L_3$ are line bundles on $C$ satisfying 
$L_1\otimes L_2\otimes L_3\cong \omega_C$. 
The
three dimensional complex torus $T=(\mathbb{C}^{\ast})^{\times 3}$
acts on $X$ fiberwise over $C$. 
The $T$-equivariant GW invariants
\begin{align*}
\mathrm{GW}_{0, d[C]} \in \mathbb{Q}(\lambda_1, \lambda_2, \lambda_3)
\end{align*}
can be defined via equivariant residue (e.g.~\cite{Kont}). 
Here $\lambda_i$ are the equivariant parameters
with respect to the $T$-action. 

On the other hand, there 
is a two dimensional subtorus $T_0\subseteq (\mathbb{C}^{\ast})^{3}$ which preserves the CY 4-form on $X$. 
As in the local surface case~(\ref{intro:localization}), 
we may 
define equivariant $\mathrm{DT_4}$ invariants 
\begin{align*}
\mathrm{DT}_{4}(d[C])\in\mathbb{Q}(\lambda_1, \lambda_2)
\end{align*}
as rational functions in terms of equivariant parameters of $T_0$. 
We explicitly determine $\mathrm{DT}_{4}(d[C])$ for $d\leqslant2$ (Corollary \ref{cor on deg one local curve}, \ref{DT4 for 2C}) and propose the 
following conjecture: 
\begin{conj}\label{intro: conj on local P1}\emph{(Conjecture \ref{conj on local P1})}
For any smooth projective curve $C$ and line bundles $L_i$ ($i=1,2,3$) on $C$ with $L_1 \otimes L_2 \otimes L_3 \cong \omega_C$, we have the identity
\begin{align*} \mathrm{GW}_{0,2[C]}=\mathrm{DT}_{4}(2[C])+\frac{1}{8}\mathrm{DT}_{4}([C])\in\mathbb{Q}(\lambda_1, \lambda_2) \end{align*}
after substituting $\lambda_3=-\lambda_1-\lambda_2$ in the LHS.
\end{conj}
By computing $\DT_4([2C])$ explicitly and using Mathematica, we 
prove the following: 
\begin{thm}\emph{(Theorem \ref{prop on local P1})} 
Conjecture~\ref{intro: conj on local P1} is true if 
\begin{enumerate}
\item 
$g(C)\geqslant1$; 
\item 
$g(C)=0$ and $|l_1|\leqslant10$, $|l_2|\leqslant10$ for $l_i=\deg(L_i)$
with $l_1 \geqslant l_2 \geqslant l_3$. 
\end{enumerate}
\end{thm}
Indeed the conjectural formula 
in Conjecture~\ref{intro: conj on local P1}
implies some non-trivial identities of rational functions of 
$\lambda_1, \lambda_2$. See Appendix~B. 


\subsection{Speculation for genus one GV type invariants}
For genus one,
virtual dimensions of GW moduli spaces 
on a CY 4-fold $X$ are zero, so
the GW invariants
\begin{align*}
\mathrm{GW}_{1, \beta}=\int_{[\overline{M}_{1, 0}(X, \beta)]^{\rm{vir}}}
1 \in \mathbb{Q}
\end{align*}
can be defined
without insertions.
The
invariants
\begin{align}\label{intro:n1}
n_{1, \beta} \in \mathbb{Q}
\end{align}
 are defined in~\cite{KP} by the identity
\begin{align*}
\sum_{\beta>0}
\mathrm{GW}_{1, \beta}q^{\beta}=
&\sum_{\beta>0} n_{1, \beta} \sum_{d=1}^{\infty}
\frac{\sigma(d)}{d}q^{d\beta}
+\frac{1}{24}\sum_{\beta>0} n_{0, \beta}(c_2(X))\log(1-q^{\beta}) \\
&-\frac{1}{24}\sum_{\beta_1, \beta_2}m_{\beta_1, \beta_2}
\log(1-q^{\beta_1+\beta_2}),
\end{align*}
where $\sigma(d)=\sum_{i|d}i$ and $m_{\beta_1, \beta_2}\in\mathbb{Z}$ 
are called meeting invariants which can be inductively determined by genus 
zero GW invariants. In~\cite{KP}, the invariants (\ref{intro:n1}) are also conjectured to be integers. 

It is a natural problem to 
give a sheaf-theoretic interpretation to the invariants (\ref{intro:n1}). 
In the CY 3-fold case, the current 
proposals for such invariants 
involve refined versions of DT invariants
(see~\cite{HST, KL, MT}), 
 which are not available for CY 4-folds. 
Instead, we speculate that there is an interpretation of
the invariants (\ref{intro:n1}) using the
$\DT_4$-version of 
Pandharipande-Thomas' stable pair invariants~\cite{PT}.

Namely let us 
consider the stable pair moduli space $P_{n}(X,\beta)$ which parametrizes stable pairs 
\begin{align*}
(s:\oO_{X} \to F), \ [F]=\beta, \ \chi(F)=n
\end{align*}
i.e. $F$ is a pure one dimensional sheaf and $s$ is surjective in dimension 
one. By assuming an orientability result, one can similarly define the virtual class
\begin{align}\label{PT:virtual}
[P_{n}(X,\beta)]^{\rm{vir}} \in H_{2n}(P_{n}(X,\beta),\mathbb{Z}). 
\end{align}
A difference from the CY 3-fold case is that the virtual 
dimension of (\ref{PT:virtual}) 
depends on the holomorphic Euler characteristic $n$. 
A special case is when $n=0$, where
the virtual dimension is zero, we can take its degree: 
\begin{align*}
P_{0, \beta}:=\int_{[P_0(X, \beta)]^{\rm{vir}}} 1 \in \mathbb{Z}. 
\end{align*}
We expect the above invariants to contain information on the genus one GV type invariants (\ref{intro:n1}). 
Some of our computations suggest 
the following formula  
\begin{align*}
\sum_{\beta \geqslant 0}
P_{0, \beta}q^{\beta}=
\prod_{\beta>0} M\left(q ^{\beta}\right)^{n_{1, \beta}},
\end{align*}
where $M(q)=\prod_{k\geqslant 1}(1-q^{k})^{-k}$ is the MacMahon function.
Some more analysis of the above formula 
may be pursued in a future work.

\subsection{Notation and convention}
In this paper, all varieties and schemes are defined over $\mathbb{C}$. 
For a morphism $\pi \colon X \to Y$ of schemes, 
and for $\fF, \gG \in \mathrm{D^{b}(Coh(X))}$, we denote by 
$\dR \hH om_{\pi}(\fF, \gG)$ 
the functor $\dR \pi_{\ast} \dR \hH om_X(\fF, \gG)$. 
We also denote by $\mathrm{ext}^i(\fF, \gG)$ the dimension of 
$\Ext^i_X(\fF, \gG)$. 

A class $\beta\in H_2(X,\mathbb{Z})$ is called \textit{irreducible} (resp. \textit{primitive}) if it is not the sum of two non-zero effective classes
(resp. if it is not a positive integer multiple of an effective class).

For a scheme $X$, we denote by $\chi(X)$ its topological Euler characteristic. 
For a sheaf $\fF$ on $X$, we denote by $\chi(\fF)$ its holomorphic Euler characteristic. 

\subsection{Acknowledgement}
We would like to thank Duiliu-Emanuel Diaconescu for suggesting the problem and useful discussions and Dominic Joyce for helpful comments on our
preprint. The first author is grateful to Meya Zhao for her help in using 'Mathematica' to verify our conjectures in examples. 
Y.~C.~ is supported by the Royal Society Newton International Fellowship.
D.~M.~is supported by 
NSF grants DMS-1645082 and DMS-1564458.
Y.~T.~is supported by
Grant-in Aid
for Scientific Research grant (No.~26287002)
from MEXT, Japan. 
Y.~C.~and Y.~T.~ 
were supported by World Premier International Research Center Initiative (WPI), MEXT, Japan.

\section{Definitions and conjectures}
Throughout this paper, unless stated otherwise, $X$ will always denote a smooth projective Calabi-Yau 4-fold over $\mathbb{C}$, i.e. $K_X\cong \mathcal{O}_X$. 

\subsection{GW/GV conjecture on Calabi-Yau 4-folds}
Let $\overline{M}_{g, n}(X, \beta)$
be the moduli space of genus $g$, $n$-pointed stable maps
to $X$ with curve class $\beta$.
Its virtual dimension is given by
\begin{align*}
-K_X \cdot \beta+(\dim X-3)(1-g)+n=1-g+n.
\end{align*}
For cohomology classes
\begin{align}\label{gamma}
\gamma_i \in H^{m_i}(X, \mathbb{Z}), \
1\leqslant i\leqslant n, 
\end{align}
the corresponding Gromov-Witten invariant is defined by
\begin{align}\label{GWinv}
\mathrm{GW}_{g, \beta}(\gamma_1, \ldots, \gamma_n)
=\int_{[\overline{M}_{g, n}(X, \beta)]^{\rm{vir}}}
\prod_{i=1}^n \mathrm{ev}_i^{\ast}(\gamma_i),
\end{align}
where $\mathrm{ev}_i \colon \overline{M}_{g, n}(X, \beta)\to X$
is the $i$-th evaluation map.

For $g=0$, the
virtual dimension of $\overline{M}_{0, n}(X, \beta)$
is $n+1$, and (\ref{GWinv})
is zero unless
\begin{align}\label{sum:m}
\sum_{i=1}^{n}(m_i-2)=2.
\end{align}
In analogy with the Gopakumar-Vafa conjecture for CY 3-folds \cite{GV}, Klemm-Pandharipande \cite{KP} defined invariants $n_{0, \beta}(\gamma_1, \ldots, \gamma_n)$ on CY 4-folds by the identity
\begin{align*}
\sum_{\beta>0}\mathrm{GW}_{0, \beta}(\gamma_1, \ldots, \gamma_n)q^{\beta}=
\sum_{\beta>0}n_{0, \beta}(\gamma_1, \ldots, \gamma_n) \sum_{d=1}^{\infty}
d^{n-3}q^{d\beta},
\end{align*}
and conjecture the following
\begin{conj}\emph{(\cite[Conjecture~0]{KP})}\label{KPconj}
The invariants $n_{0, \beta}(\gamma_1, \ldots, \gamma_n)$
are integers.
\end{conj}

For $g\geqslant2$, GW invariants vanish for dimensional reasons, so a GW/GV type integrality conjecture on CY 4-folds makes sense only for genus 0 and 1.
We will propose a sheaf-theoretic definition of  GV-type integral invariants 
using $\DT_4$ invariants~\cite{CL, BJ}. 

\subsection{Review of $\mathrm{DT_ 4}$ invariants}\label{sec:review}
Before stating our proposal, we first review the framework for $\mathrm{DT_ 4}$ invariants.
We fix an ample divisor $\omega$ on $X$
and take a cohomology class
$v \in H^{\ast}(X, \mathbb{Q})$.

The coarse moduli space $M_{\omega}(v)$
of $\omega$-Gieseker semistable sheaves
$E$ on $X$ with $\ch(E)=v$ exists as a projective scheme.
We always assume that
$M_{\omega}(v)$ is a fine moduli space, i.e.
any point $[E] \in M_{\omega}(v)$ is stable and
there is a universal family
\begin{align}\label{universal}
\eE \in \Coh(X \times M_{\omega}(v)).
\end{align}

In~\cite{CL, BJ}, under certain hypotheses,
the authors construct 
a $\mathrm{DT}_{4}$ virtual
class
\begin{align}\label{virtual}
[M_{\omega}(v)]^{\rm{vir}} \in H_{2-\chi(v, v)}(M_{\omega}(v), \mathbb{Z}), \end{align}
where $\chi(-,-)$ is the Euler pairing.
Notice that this class will not necessarily be algebraic.

Roughly speaking, in order to construct such a class, one chooses at
every point $[E]\in M_{\omega}(v)$, a half-dimensional real subspace
\begin{align*}\Ext_{+}^2(E, E)\subset \Ext^2(E, E)\end{align*}
of the usual obstruction space $\Ext^2(E, E)$, on which the quadratic form $Q$ defined by Serre duality is real and positive definite. 
Then one glues local Kuranishi-type models of form 
\begin{equation}\kappa_{+}=\pi_+\circ\kappa: \Ext^{1}(E,E)\to \Ext_{+}^{2}(E,E),  \nonumber \end{equation}
where $\kappa$ is a Kuranishi map of $M_{\omega}(v)$ at $E$ and $\pi_+$ is the projection 
according to the decomposition $\Ext^{2}(E,E)=\Ext_{+}^{2}(E,E)\oplus\sqrt{-1}\cdot\Ext_{+}^{2}(E,E)$.  \\

In \cite{CL}, local models are glued in three special cases: 
\begin{enumerate}
\item when $M_{\omega}(v)$ consists of locally free sheaves only; 
\item  when $M_{\omega}(v)$ is Kuranishi smooth, i.e. local Kuranishi maps $\kappa$'s vanish; 
\item when $M_{\omega}(v)$ is a shifted cotangent bundle of a derived smooth scheme. 
\end{enumerate}
And the corresponding virtual classes are constructed using either gauge theory or algebro-geometric perfect obstruction theory in the sense of Behrend-Fantechi \cite{BF} and Li-Tian \cite{LT}.

The general gluing construction is due to Borisov-Joyce \cite{BJ}\,\footnote{One needs to assume that $M_{\omega}(v)$ can be given a $(-2)$-shifted symplectic structure as in Claim 3.29 \cite{BJ} to apply their constructions.}, based on Pantev-T\"{o}en-Vaqui\'{e}-Vezzosi's theory of shifted symplectic geometry \cite{PTVV} and Joyce's theory of derived $C^{\infty}$-geometry.
The corresponding virtual class is constructed using Joyce's
D-manifold theory (a machinery similar to Fukaya-Oh-Ohta-Ono's theory of Kuranishi space structures used in defining Lagrangian Floer theory).

In this paper, all computations and examples will only involve the virtual class constructions in situations (2), (3), mentioned above. We briefly 
review them as follows:  
\begin{itemize}
\item When $M_{\omega}(v)$ is Kuranishi smooth, the obstruction sheaf $Ob\to M_{\omega}(v)$ is a vector bundle endowed with a quadratic form $Q$ via Serre duality. Then the $\DT_4$ virtual class is given by
\begin{equation}[M_{\omega}(v)]^{\rm{vir}}=\mathrm{PD}(e(Ob,Q)).   \nonumber \end{equation}
Here $e(Ob, Q)$ is the half-Euler class of 
$(Ob,Q)$ (i.e. the Euler class of its real form $Ob_+$), 
and $\mathrm{PD}(-)$ is its 
Poincar\'e dual. 
Note that the half-Euler class satisfies 
\begin{align*}
e(Ob,Q)^{2}&=(-1)^{\frac{\mathrm{rk}(Ob)}{2}}e(Ob),  \textrm{ }\mathrm{if}\textrm{ } \mathrm{rk}(Ob)\textrm{ } \mathrm{is}\textrm{ } \mathrm{even}, \\
 e(Ob,Q)&=0, \textrm{ }\mathrm{if}\textrm{ } \mathrm{rk}(Ob)\textrm{ } \mathrm{is}\textrm{ } \mathrm{odd}. 
\end{align*}
\item When $M_{\omega}(v)$ is a shifted cotangent bundle of a derived smooth scheme, roughly speaking, this means that at any closed point $[F]\in M_{\omega}(v)$, we have Kuranishi map of type
\begin{equation}\kappa \colon
 \Ext^{1}(F,F)\to \Ext^{2}(F,F)=V_F\oplus V_F^{*},  \nonumber \end{equation}
where $\kappa$ factors through a maximal isotropic subspace $V_F$ of $(\Ext^{2}(F,F),Q)$. Then the $\DT_4$ virtual class of $M_{\omega}(v)$ is, 
roughly speaking, the 
virtual class of the perfect obstruction theory formed by $\{V_F\}_{F\in M_{\omega}(v)}$. 
When $M_{\omega}(v)$ is furthermore smooth as a scheme, 
then it is
simply the Euler class of the vector bundle 
$\{V_F\}_{F\in M_{\omega}(v)}$ over $M_{\omega}(v)$. 
\end{itemize}
${}$ \\
\textbf{On orientations}.
To construct the above virtual class (\ref{virtual}) with coefficients in $\mathbb{Z}$ (instead of $\mathbb{Z}_2$), we need an orientability result 
for $M_{\omega}(v)$, which is stated as follows.
Let  
\begin{equation}\label{det line bdl}
 \lL:=\mathrm{det}(\dR \hH om_{\pi_M}(\eE, \eE))
 \in \Pic(M_{\omega}(v)), \quad  
\pi_M \colon X \times M_{\omega}(v)\to M_{\omega}(v),
\end{equation}
be the determinant line bundle of $M_{\omega}(v)$, equipped with a symmetric pairing $Q$ induced by Serre duality.  An \textit{orientation} of 
$(\mathcal{L},Q)$ is a reduction of its structure group (from $O(1,\mathbb{C})$) to $SO(1, \mathbb{C})=\{1\}$; in other words, we require a choice of square root of the isomorphism
\begin{equation}\label{Serre duali}Q: \lL\otimes \lL \to \oO_{M_{\omega}(v)}  \end{equation}
to construct virtual class (\ref{virtual}).
An existence result of orientations is proved in \cite[Theorem 2.2]{CL2} for CY 4-folds $X$ such that $\mathrm{Hol}(X)=SU(4)$ and $H^{\rm{odd}}(X,\mathbb{Z})=0$. 
Notice that, if orientations exist, their choices form a torsor for $H^{0}(M_{\omega}(v),\mathbb{Z}_2)$. 

At the moment, there is no canonical choice of orientation for defining the virtual class (\ref{virtual}). 
We restrict to the case of CY 4-fold $X$ with $\mathrm{Hol}(X)=SU(4)$ and $H^{\rm{odd}}(X,\mathbb{Z})=0$ and explain
how to obtain a deformation invariant virtual class (hence also the corresponding invariants).

Let $\mathfrak{M}$ be the space of all complex structures on $X$. 
We take a finite type connected open subset $\mathcal{U}$ of $\mathfrak{M}$, and
a relative ample line bundle $\mathcal{O}_{X_\mathcal{U}}(1)$. Let $\{M_{\mathcal{O}_{X_t}(1)}(v)\}_{t\in\mathcal{U}}$ be a family of fine Gieseker moduli space (for some $v\in H^{*}(X,\mathbb{Q})$).
We can do a family of Seidel-Thomas twists which identifies 
$\{M_{\mathcal{O}_{X_t}(1)}(v)\}_{t\in\mathcal{U}}$ with a family of moduli space $\{\mathcal{M}_{si,t}^{bdl}\}_{t\in\mathcal{U}}$ of simple holomorphic structures on a complex vector bundle $E$.

Recall that the proof of \cite[Theorem 2.2] {CL2} shows the orientation of $(\mathcal{L},Q)$ over 
$M_{\mathcal{O}_{X}(1)}(v)$ is given as follows: firstly use S-T twists to identify the moduli space to a moduli of simple holomorphic structures on a complex bundle $E$ and then get an induced orientation from the space $\widetilde{\mathcal{B}}_{E',X}$ of gauge equivalence classes of framed connections on $E'=E\oplus(\det E)^{-1}\oplus \mathbb{C}^{N}$ with $N\gg0$. Note that $\widetilde{\mathcal{B}}_{E',X}$ does not depend on the choice of complex structures on $X$, and we know $\pi_1(\widetilde{\mathcal{B}}_{E',X})=0$ by \cite[Theorem 2.1] {CL2}. So we can choose orientation of $(\mathcal{L},Q)$ consistently over the family $\mathcal{U}$ such that the virtual class is invariant under deformation of complex structures.

\subsection{One-dimensional stable sheaves and the genus zero conjecture}
Let us take $\beta \in H_2(X, \mathbb{Z})\cong H^{6}(X,\mathbb{Z})$. By taking the cohomology class
\begin{align*}
v=(0, 0, 0, \beta, 1) \in H^0(X) \oplus H^2(X) \oplus H^4(X) \oplus H^6(X)\oplus H^8(X),
\end{align*}
we set
\begin{align*}
M_{\beta}=M_{\omega}(0, 0, 0, \beta, 1).
\end{align*}
\begin{rmk}\label{rmk on stability}
Note that $M_{\beta}$ is
the moduli space of one-dimensional
sheaves $E$'s on $X$ satisfying the following:
for any $0\neq E' \subsetneq E$, we have $\chi(E')\leqslant 0$.
In particular, it is
independent of the choice of $\omega$ and a fine moduli space.
\end{rmk}
Since $\chi(v, v)=0$ in this case, we have
\begin{align*}
[M_{\beta}]^{\rm{vir}} \in H_2(M_{\beta}, \mathbb{Z}).
\end{align*}
We define insertions as follows:
for a class $\gamma \in H^m(X, \mathbb{Z})$, set
\begin{equation}\label{tau gamma}
\tau(\gamma):=\pi_{M\ast}(\pi_X^{\ast}\gamma \cup\ch_3(\eE) )
\in H^{m-2}(M_{\beta}, \mathbb{Z}). 
\end{equation}
Here $\pi_X$, $\pi_M$ are projections from $X \times M_{\beta}$
onto corresponding factors, and $\ch_3(\eE)$ is the
Poincar\'e dual to the fundamental class of the universal sheaf (\ref{universal}).
For integral classes (\ref{gamma}),
we define
\begin{align}\label{DT4}
\mathrm{DT}_{4}(\beta\mid \gamma_1, \ldots, \gamma_n)=
\int_{[M_{\beta}]^{\rm{vir}}}
\prod_{i=1}^{n} \tau(\gamma_i)
\in
\mathbb{Z}.
\end{align}
Note that (\ref{DT4}) is zero unless (\ref{sum:m}) holds.
We propose the following conjecture, which
gives a sheaf theoretic interpretation of
Conjecture~\ref{KPconj}.
\begin{conj}\label{conj:GW/GV}\emph{(Genus 0)}
We have the identity
\begin{align*}
n_{0,\beta}(\gamma_1, \ldots, \gamma_n)=
\mathrm{DT}_{4}(\beta\mid \gamma_1, \ldots, \gamma_n),
\end{align*}
for certain choice of orientation in defining the RHS. 

In particular, we have the multiple cover formula
\begin{align*}
\mathrm{GW}_{0, \beta}(\gamma)=
\sum_{k|\beta}\frac{1}{k^{2}}\cdot\mathrm{DT}_{4}(\beta/k \mid \gamma).
\end{align*}
\end{conj}
\begin{rmk}
For the insertions, if $m_n=2$ in (\ref{gamma}), we have
\begin{align*}
n_{0,\beta}(\gamma_1, \ldots, \gamma_n) &=
(\beta \cdot \gamma_n) \cdot n_{0,\beta}(\gamma_1, \ldots, \gamma_{n-1}), \\
\mathrm{DT}_{4}(\beta\mid \gamma_1, \ldots, \gamma_n)
&=(\beta \cdot \gamma_n) \cdot
\mathrm{DT}_{4}(\beta\mid \gamma_1, \ldots, \gamma_{n-1}).
\end{align*}
Therefore we may assume that $m_i \geqslant 3$ for all $i$
in Conjecture~\ref{conj:GW/GV}.
By (\ref{sum:m}), there are
two possibilities
\begin{itemize}
\item $n=1$ and $m_1=4$,
\item $n=2$ and $m_1=m_2=3$.
\end{itemize}
In particular, when $H^{\rm{odd}}(X, \mathbb{Z})=0$,
we only need to consider the first case.
\end{rmk}

\subsection{Heuristic explanation of the conjecture}\label{subsec:geometric}
In this subsection, we give a heuristic argument to explain why we expect Conjecture \ref{conj:GW/GV} to be true.  In this discussion,
we ignore questions of orientation.

Let $X$ be an 'ideal' $\mathrm{CY_{4}}$ in the sense that all curves inside are smooth of expected dimensions, i.e.
\begin{enumerate}
\item
 any rational curve in $X$
comes with a compact 1-dimensional smooth family of embedded rational curves, whose general member is smooth with 
normal bundle $\mathcal{O}_{\mathbb{P}^{1}}(-1,-1,0)$.  
\item
any elliptic curve $E$ in $X$ is smooth, super-rigid, i.e. 
the normal bundle is 
$L_1 \oplus L_2 \oplus L_3$
for general degree zero line bundle $L_i$ on $E$
satisfying $L_1 \otimes L_2 \otimes L_3=\oO_E$. 
Furthermore any two elliptic curves are 
disjoint. 

\item
there is no curve in $X$ with genus $g\geqslant2$.
\end{enumerate}

Let $C$ be a rational curve in $X$ with $[C]=\beta\in H_{2}(X,\mathbb{Z})$, and $\{C_{t}\}_{t\in T}$ be the 1-dimensional family $C$ sits in. Any one-dimensional stable sheaf $F\in M_{\beta}$ supported on $C_{t}$
is $\mathcal{O}_{C_{t}}$ for some curve $C_{t}$. 
For a general $t \in T$ such that 
$\mathcal{N}_{C_{t}/X}\cong\mathcal{O}_{\mathbb{P}^{1}}(-1,-1,0)$, 
there exists canonical isomorphisms
\begin{align*}
&\Ext^{1}_X(F,F)\cong H^{0}(C_{t},\mathcal{N}_{C_{t}/X})\cong\mathbb{C},  \\
&\Ext^{2}_X(F,F)\cong H^{1}(C_{t},\mathcal{N}_{C_{t}/X})\oplus H^{1}(C_{t},\mathcal{N}_{C_{t}/X})^{\vee}=0. 
\end{align*}
Hence the local contribution of $C_{t}$ to the
$\mathrm{DT_{4}}$ virtual class $[M_{\beta}]^{\rm{vir}}$ is the fundamental class of the family $[T]$.

Let $E$ be an elliptic curve in $X$
with $[E]=\beta\in H_{2}(X,\mathbb{Z})$. 
Then any one dimensional stable sheaf $F \in M_{\beta}$ supported on 
$E$ is a line 
bundle on $E$. 
A similar calculation shows that we have 
\begin{align*}
&\Ext^{1}_X(F,F)\cong H^{0}(E,\mathcal{N}_{E/X})\oplus H^{0}(E,\mathcal{O}_{E})
\cong\mathbb{C}, \\
&\Ext^{2}_X(F,F)\cong H^{1}(E,\mathcal{N}_{E/X})\oplus
H^{1}(E,\mathcal{N}_{E/X})^{\vee}=0.
\end{align*}
Hence the local contribution of $E$ to $\mathrm{DT_{4}}$ virtual class $[M_{\beta}]^{\rm{vir}}$
is the fundamental class of the Picard variety $\Pic(E)\cong E$.

Let us take $\gamma \in H^4(X, \mathbb{Z})$
and consider the $\DT_4$ invariant
\begin{align}\label{DT4:ideal}
\DT_4(\beta \mid \gamma)=\int_{[M_{\beta}]^{\rm{vir}}}
\tau(\gamma). 
\end{align} 
Since the insertion $\tau(\gamma)$ 
imposes a codimension one constraint on the 
deformation space of curves in $X$, 
an elliptic curve $E \subset X$ does not 
contribute to (\ref{DT4:ideal})
as it is rigid.  
By the above argument, we have 
\begin{align*}
\int_{[M_{\beta}]^{\rm{vir}}}\tau(\gamma)=\sum_{i}\gamma\cdot[C_{T_{i}}], 
\end{align*}
where $[C_{T_{i}}]\in H_{4}(X,\mathbb{Z})$ denotes the fundamental class of the $T_{i}$ family of rational curves. This heuristic argument confirms Conjecture \ref{conj:GW/GV} as both $n_{0,\beta}(\gamma)$ and $\mathrm{DT_{4}}(\beta\textrm{ } |\textrm{ } \gamma)$ are (virtually) enumerating rational curves of class $\beta$ incident to cycle dual to $\gamma$.

\section{Compact examples}
In this section, we verify Conjecture \ref{conj:GW/GV}
 for certain compact Calabi-Yau 4-folds.

\subsection{Elliptic fibrations}
For $Y=\mathbb{P}^3$, we take general elements
\begin{align*}
u \in H^0(Y, \oO_Y(-4K_Y)), \
v \in H^0(Y, \oO_Y(-6K_Y)).
\end{align*}
Let $X$
be a CY 4-fold with an elliptic fibration
\begin{align*}
\pi \colon X \to Y
\end{align*}
given by the equation
\begin{align*}
zy^2=x^3 +uxz^2+vz^3
\end{align*}
in the $\mathbb{P}^2$-bundle
\begin{align*}
\mathbb{P}(\oO_Y(-2K_Y) \oplus \oO_Y(-3K_Y) \oplus \oO_Y) \to Y,
\end{align*}
where $[x:y:z]$ are homogeneous coordinates for the above
projective bundle. A general fiber of
$\pi$ is a smooth elliptic curve, and any singular
fiber is either a nodal or cuspidal plane curve.
Moreover, $\pi$ admits a section $\iota$ whose image
correspond to the fiber point $[0: 1: 0]$.

Let $h$ be a hyperplane in $\mathbb{P}^3$, $f$ be a general fiber of $\pi \colon X\rightarrow Y$ and set
\begin{align}\label{div:BE}
B=\pi^{\ast}h, \ E=\iota(\mathbb{P}^3)\in H_{6}(X,\mathbb{Z}).
\end{align}
We consider the moduli space $M_{r[f]}$ of 
one dimensional stable sheaves on $X$ in the 
multiple fiber class $r[f]$.
\begin{lem}\label{lem:fib:isom}
For any $r \in \mathbb{Z}_{\geqslant 1}$, we have an isomorphism
$M_{r[f]} \cong X$, 
under which the virtual
class of $M_{r[f]}$ is given by
\begin{align}\label{Mvir}
[M_{r[f]}]^{\rm{vir}}=\pm \mathrm{PD}(c_3(X)) \in H_2(X, \mathbb{Z}),
\end{align}
where the sign corresponds to the choice of an orientation in defining the LHS.
\end{lem}
\begin{proof}
By stability, any one dimensional stable
sheaf $\eE$ with $[\eE]=r[f]$
satisfies $\Hom(\eE, \eE)=\mathbb{C}$, 
hence by Lemma~\ref{lem:known}
it  is scheme theoretically
supported on a fiber of $\pi$.
Therefore $M_{r[f]}$ is identified with the
$\pi$-relative moduli space of stable sheaves on $X$.
By~\cite[Theorem 2.1]{B-M2},
$M_{r[f]}$ is a smooth CY 4-fold which is derived equivalent to $X$.
There are rational maps
\begin{align*}
M_{r[f]} \stackrel{\phi_1}{\dashrightarrow}
M_{[f]} \stackrel{\phi_2}{\dashleftarrow} X,
\end{align*}
where $\phi_1$ sends a stable sheaf $\eE$ on
$\pi^{-1}(p)$ for a general point $p\in Y$ to
$\det (\eE)$, and
$\phi_2$ sends $x \in \pi^{-1}(p)$
to $I_x^{\vee}$
where $I_x \subset \oO_{\pi^{-1}(p)}$ is the
ideal sheaf of $x$ in $\pi^{-1}(p)$. It is well-known that
$\phi_1$, $\phi_2$ are isomorphisms
on a general fiber of $\pi$ by Atiyah's work,
hence they are birational maps.
Therefore, $M_{r[f]}$ and $X$ are connected by a finite number of flops
by~\cite{Kawaflo}.
On the other hand, the Picard number of $X$ is two by
the weak Lefschetz theorem and the
extremal rays of its nef cone are given by the divisors
(\ref{div:BE}).
Therefore $X$ does not admit any flop, and
the birational map 
$\phi_2^{-1} \circ \phi_1$ extends to an 
isomorphism
\begin{align*}
\phi_2^{-1} \circ \phi_1 \colon 
M_{r[f]} \stackrel{\cong}{\to} X.
\end{align*}

Using the derived equivalence between $M_{r[f]}$ and $X$ 
proved in~\cite{B-M2},
the deformation-obstruction spaces (with Serre duality pairing) on $M_{r[f]}$
are identified with those on $X$ 
viewed as a moduli space of
skyscraper sheaves $\{\oO_x\}_{x \in X}$.
Therefore the identity (\ref{Mvir}) holds
from~\cite[Proposition~7.17]{CL}.
\end{proof}
Here we used the following lemma, which is well-known: 
\begin{lem}\label{lem:known}
Let $f \colon X \to Y$ be a morphism of schemes 
and suppose we have $E \in \Coh(X)$ which is set theoretically 
supported on 
$f^{-1}(p)$ for some $p \in Y$
and satisfies $\Hom(E, E)=\mathbb{C}$. 
Then $E$ is scheme-theoretically supported on 
$f^{-1}(p)$. 
\end{lem}
\begin{proof}
Let $m_{p} \subset \oO_Y$ be the ideal sheaf of $p$. 
For $u \in m_p$, 
the multiplication by $u$ defines the morphism 
\begin{align}\label{mult:u}
u \colon E \to E.
\end{align}
Since $\Hom(E, E)=\mathbb{C}$, 
the morphism (\ref{mult:u}) is either isomorphism or zero map. 
As $E$ is set theoretically supported on $f^{-1}(p)$, 
some compositions of (\ref{mult:u}) must be a zero map. 
Therefore (\ref{mult:u}) is a zero map
for any $u\in m_p$, which implies that $E$ is 
$\oO_{f^{-1}(p)}$-module. 
\end{proof}

We now verify Conjecture~\ref{conj:GW/GV} for multiple fiber classes.
\begin{prop}\label{prop on elliptic fib}
Conjecture~\ref{conj:GW/GV} is true for $\beta=r[f]$,
$\gamma=B^2$ or $B\cdot E$.
\end{prop}
\begin{proof}
Let $\eE$ be the universal sheaf on $X \times M_{r[f]}$.
Under the isomorphism $M_{r[f]} \cong X$ in 
Lemma~\ref{lem:fib:isom}, we have
\begin{align*}
\ch_3(\eE)=r[X \times_Y X] \in H^6(X \times X, \mathbb{Z}).
\end{align*}
Together with the identity (\ref{Mvir}),
for $\gamma \in H^4(X, \mathbb{Z})$ we obtain
\begin{align*}
\mathrm{DT}_{4}(r[f]\textrm{ }\big| \textrm{ }\gamma)=\pm
r \int_{X} \pi^{\ast}\pi_{\ast}\gamma \cup c_3(X).
\end{align*}
For $\gamma=B^2$,
we have $\pi_{\ast}\gamma=0$, hence
\begin{align}\label{B2}
\mathrm{DT}_{4}(r[f]\textrm{ }\big| \textrm{ }B^2)=0, \
r \in \mathbb{Z}_{\geqslant 1}.
\end{align}
For $\gamma=B \cdot E$, we have
$\pi^{\ast}\pi_{\ast}\gamma=B$ and
\begin{align}\label{BE}
\mathrm{DT}_{4}(r[f]\textrm{ }\big| \textrm{ }B \cdot E)=\pm
r \int_{X} B \cup c_3(X)=
\pm\, 960r.
\end{align}
In both cases (\ref{B2}), (\ref{BE}), the
results agree (for a suitable choice of sign) with $n_{r[f]}(\gamma)$ (which is $0$, $960r$ respectively) given
in~\cite[Table~7]{KP}\footnote{In \cite[Table~7]{KP}, genus zero GW invariants are computed by Picard-Fuchs equations through mirror symmetry. Fortunately, our $\mathrm{CY_{4}}$ is a hypersurface in a toric variety, mirror principle has been verified in this case by the works of Lian-Liu-Yau and Givental. }.\end{proof}

\subsection{Quintic fibrations}
We consider a compact CY 4-fold $X$ which admits a quintic 3-fold fibration structure
\begin{equation}\pi: X\rightarrow \mathbb{P}^{1},  \nonumber \end{equation}
i.e. $\pi$ is a proper morphism whose general fiber is a smooth quintic 3-fold $Y\subseteq\mathbb{P}^{4}$.
Examples of such CY 4-folds 
include resolutions of degree 10 orbifold hypersurface in $\mathbb{P}^{5}(1,1,2,2,2,2)$ and hypersurface of bidegree $(2,5)$ in $\mathbb{P}^{1}\times \mathbb{P}^{4}$ (see \cite{KP}).
We will explain some part of \cite[Table~4, Table~6]{KP} for these two examples.   
Let $h$ be the hyperplane class of $\mathbb{P}^1$ and $B=\pi^{\ast}h$ be the fiber class of $\pi$, $F\in H_6(X)$ be a divisor Poincar\'e dual to degree one curve in the quintic fiber. 
Notice that the genus zero integral invariants $n_{0,\beta}^{2}$'s (w.r.t. insertion $B\cdot F$) for both tables are the same
when $\beta\in H_2(Y)\subseteq H_2(X)$. Moreover, under the identification $H_2(Y)\cong\mathbb{Z}$, $\beta\mapsto d$, these numbers are $d$ times the degree $d$, genus zero Gopakumar-Vafa invariants of quintic 3-fold $Y$ \cite{GV} which are conjecturally the same as $\mathrm{DT}_{3}$ invariants for one dimensional stable sheaves on $Y$ \cite{Katz}. Hence, Conjecture \ref{conj:GW/GV} for this case may be reduced to a relation between $\mathrm{DT}_{4}$ invariants of $X$ and $\mathrm{DT}_{3}$ invariants of $Y$.

More generally, we expect the following conjecture: 
\begin{conj}\label{conj on CY4/CY3}
Let $X$ be a projective CY 4-fold which admits a CY 3-fold fibration\begin{equation}\pi \colon X\rightarrow C  \nonumber \end{equation} 
over a smooth curve $C$.

Let $X_{p}=\pi^{-1}(p)$ be a general fiber and $\beta\in\mathrm{Im}(i_*:H_2(X_p)\to H_{2}(X))$. Then for an ample divisor $H\subseteq X$, we have
\begin{equation}\mathrm{DT}_{4}(\beta\mid H\cdot X_{p})=(H\cdot\beta)\cdot\sum_k\mathrm{DT}_{3}(\beta_k), \nonumber \end{equation}
for certain choice of orientation in defining the LHS.
Here $\{\beta_k\}\subseteq H_2(X_p)$ are (finitely many) non-zero effective curve classes of $X_p$ which map to $\beta$. $\mathrm{DT}_{3}(\beta_k)$ is the DT invariant \cite{Thomas} for 1-dimensional stable sheaves $\mathcal{F}\in \Coh(X_{p})$ 
with $[\mathcal{F}]=\beta_k$ and $\chi(\fF)=1$.
\end{conj}
When $\pi \colon X\to C$ is a trivial CY 3-fold fibration, the above conjecture will be proved in Corollary \ref{g=0 product of elliptic curve and CY3}. For the general case, we give a heuristic explanation as follows.

Note that the fibration $\pi \colon X\to C$ induces a fibration on the moduli space 
\begin{align*}
\pi_M \colon M_{\beta}\to C.
\end{align*}
As any
one dimensional stable sheaf on $X$ is scheme theoretically supported on some fiber (see Lemma~\ref{lem:known}), 
the fiber $\pi_M^{-1}(p)$ is regarded as 
the moduli space of stable sheaves on $X_p$. 
For a 'general'  $p\in C$, ideally, the moduli space $\pi_{M}^{-1}(p)$ is smooth (Kuranishi maps are zero) of expected dimension 0. We take $\iota_*\mathcal{E}\in\pi_{M}^{-1}(p)$ and have
\begin{align*}
&\Ext^{1}_{X}(\iota_*\mathcal{E},\iota_*\mathcal{E})\cong \Ext^{1}_{X_p}(\mathcal{E},\mathcal{E})\oplus \Ext^{0}_{X_p}(\mathcal{E},\mathcal{E})\cong \mathbb{C}, \\
&
\Ext^{2}_{X}(\iota_*\mathcal{E},\iota_*\mathcal{E})\cong \Ext^{2}_{X_p}(\mathcal{E},\mathcal{E})\oplus \Ext^{2}_{X_p}(\mathcal{E},\mathcal{E})^{\vee}=0.
\end{align*}
Then for a small neighborhood $U(p)\subseteq C$, there is an isomorphism 
\begin{align*}
\pi_M^{-1}(U(p))\cong \pi_{M}^{-1}(p)\times U(p).
\end{align*}
In \cite{BJ}, to define virtual class of $M_\beta$, local models
of type $\kappa_{+}: \Ext^1(E,E)\to \Ext_{+}^{2}(E,E)$ are glued using
partition of unity.
Hence, the above local model
\begin{equation}\kappa_+=0: \Ext^{1}_{X}(\iota_*\mathcal{E},\iota_*\mathcal{E})\to
\Ext^{2}_{X}(\iota_*\mathcal{E},\iota_*\mathcal{E})  \nonumber \end{equation}
for closed subset $\pi_M^{-1}(p)\subseteq M_{\beta}$ can be glued over $M_\beta$. Then the representing manifold of $[M_\beta]^{\rm{vir}}$, when viewed as a submanifold of $M_\beta$, can be chosen to be $\pi_M^{-1}(U(p))$ (which is homeomorphic to $\coprod_{i=1}^{\sum_k\mathrm{DT}_{3}(\beta_k)}U(p)$) near $\pi_M^{-1}(p)$.
Note that the insertion $\tau: H^ {4}(X)\to H^{2}(M_\beta)$ satisfies 
\begin{align*}
\tau(H\cdot X_{p})=(H\cdot\beta)\,[\pi_M^{-1}(p)].
\end{align*}
Hence, we have
\begin{equation}\mathrm{DT}_{4}(\beta\mid H\cdot X_{p})=\int_{[M_\beta]^{\rm{vir}}} \tau(H\cdot X_{p}) =(H\cdot\beta)\cdot\sum_k\mathrm{DT}_{3}(\beta_k). 
\nonumber \end{equation}
\begin{cor}\label{prop g=0 quintic fib}
Assuming Conjecture \ref{conj on CY4/CY3}, then Conjecture \ref{conj:GW/GV} is true for $\beta\in H_2(Y)\subseteq H_2(X)$, $\gamma=B\cdot F$ in the quintic fibration examples in Table 4 and Table 6 of \cite{KP} if and only if the ($CY_{3}$) genus zero Gopakumar-Vafa/Donaldson-Thomas conjecture \cite{Katz} is true for $\beta\in H_2(Y)$. Here $Y$ is the quintic 3-fold realized as a general fiber of $\pi:X\to\mathbb{P}^{1}$.
\end{cor}
\begin{proof}
A similar proof will be given after Corollary \ref{g=0 product of elliptic curve and CY3}.
\end{proof}

\subsection{Product of elliptic curves and Calabi-Yau 3-folds}
Sitting in between examples given by elliptic and quintic fibrations, in this subsection, we consider a CY 4-fold of type $X=Y\times E$, where $Y$ is a projective CY 3-fold and $E$ is an elliptic curve. We will show our Conjecture \ref{conj:GW/GV} for CY 4-folds is consistent with the genus zero Gopakumar-Vafa/Donaldson-Thomas conjecture \cite{Katz} for CY 3-folds.

We pick a reference point $0\in E$ and embed $\iota \colon 
Y\hookrightarrow X$ via
$y \mapsto (y, 0)$. 
We take a curve class
\begin{align*}
\beta\in H_{2}(Y,\mathbb{Z}) \stackrel{\iota}{\hookrightarrow}
 H_{2}(X,\mathbb{Z})
\end{align*}
 and consider the moduli spaces 
\begin{align*}
M_{X,\beta}=M(0,0,0,\beta,1), \ 
M_{Y,\beta}=M(0,0,\beta,1)
\end{align*}
 of 1-dimensional stable sheaves on $X$ and $Y$
respectively. 
\begin{lem}\label{DT4/DT3}
There exists an isomorphism
$M_{X,\beta}\cong M_{Y,\beta}\times E$
under which the $\mathrm{DT_{4}}$ virtual class satisfies
\begin{equation}[M_{X,\beta}]^{\rm{vir}}=\deg[M_{Y,\beta}]^{\rm{vir}}\cdot[E],  \nonumber \end{equation}
for certain choice of orientation in defining the LHS. 
Here $[M_{Y,\beta}]^{\rm{vir}}$ is the $\mathrm{DT_{3}}$ virtual class.
\end{lem}
\begin{proof}
By Lemma~\ref{lem:known}, 
any stable sheaf $F\in M_{X,\beta}$ is scheme theoretically supported on $Y\times\{t\}$ for some $t\in E$, i.e. 
it is written as $F=\iota_{t*}\mathcal{E}$ for some $\mathcal{E}\in M_{Y,\beta}$ and $\iota_t \colon Y=Y\times\{t\} \hookrightarrow X$.
From the spectral sequence 
\begin{align*}
\Ext^{*}_{Y}(\mathcal{E},\wedge^{*}\nN_{Y/X}\otimes\mathcal{E})\Rightarrow \Ext^{*}_{X}(\iota_{t*}\mathcal{E},\iota_{t*}\mathcal{E}),
\end{align*}
and $\nN_{Y/X} \cong \oO_Y$, 
we have canonical isomorphisms
\begin{align*}
&\Ext^{1}_{X}(F,F)\cong \Ext^{1}_{Y}(\mathcal{E},\mathcal{E})\oplus \mathbb{C},
\\
&\Ext^{2}_{X}(F,F)\cong \Ext^{2}_{Y}(\mathcal{E},\mathcal{E})\oplus \Ext^{2}_{Y}(\mathcal{E},\mathcal{E})^{\vee},
\end{align*}
under which $\Ext^{2}_{Y}(\mathcal{E},\mathcal{E})$ is a maximal isotropic subspace of $(\Ext^{2}_{X}(F,F),Q_{\rm{Serre}})$. Moreover there exists a Kuranishi map $\kappa:\Ext^{1}_{X}(F,F)\rightarrow \Ext^{2}_{X}(F,F)$
for $M_{X,\beta}$ at $F$ of type
\begin{align*}
\kappa \colon \Ext^{1}_{Y}(\mathcal{E},\mathcal{E})\oplus \mathbb{C} 
&\rightarrow \Ext^{2}_{Y}(\mathcal{E},\mathcal{E})\oplus \Ext^{2}_{Y}(\mathcal{E},\mathcal{E})^{\vee}, \\
\kappa(x,t)&=(\kappa_{Y}(x),0), 
\end{align*}
where $\kappa_{Y}$ is a Kuranishi map of $M_{Y,\beta}$ at $\mathcal{E}$. Hence the map
\begin{align*}
\phi \colon 
 M_{Y,\beta}\times E \rightarrow M_{X,\beta}, \quad 
\phi(\mathcal{E},t) =(\iota_{t})_{*}\mathcal{E} 
\end{align*}
is an isomorphism. 

Similar to \cite[Theorem 6.5]{CL}, \cite[Theorem 1.6]{CL3}, under the above isomorphism, we have
\begin{equation}[M_{X,\beta}]^{\rm{vir}}=\deg[M_{Y,\beta}]^{\rm{vir}}\cdot[E],  \nonumber \end{equation}
for certain choice of orientation, where $[M_{Y,\beta}]^{\rm{vir}}\in A_{0}(M_{Y,\beta})$ is the $\mathrm{DT_{3}}$ virtual class.
\end{proof}
\begin{cor}\label{g=0 product of elliptic curve and CY3}
Let $X=Y\times E$ be a product of a projective $\mathrm{CY_{3}}$ with an elliptic curve. Fix a point $0\in E$ and denote $Y=Y\times\{0\}\subseteq X$.
Then for any $\beta\in H_{2}(Y)\subseteq H_{2}(X)$ and divisor $H\subseteq X$, we have
\begin{equation}\label{DT4/DT3 invs}\mathrm{DT}_{4}(\beta\mid H\cdot Y)=\mathrm{DT}_{3}(\beta)\cdot(H\cdot\beta), \end{equation}
for certain choice of orientation in defining the LHS.
Here $\mathrm{DT}_{3}(\beta)=\deg[M_{Y,\beta}]^{\rm{vir}}$ is the $\mathrm{DT_{3}}$ invariant (see Appendix~\ref{append:CY3}). 

In particular, Conjecture \ref{conj:GW/GV} is true for $\beta\in H_{2}(Y)\subseteq H_{2}(X)$ and $\gamma=H\cdot Y$ (for any divisor $H\subseteq X$) if and only if the ($CY_{3}$) genus zero Gopakumar-Vafa/Donaldson-Thomas 
conjecture~\cite{Katz}
(see Conjecture~\ref{conj:katz})
is true for $\beta\in H_{2}(Y)$. 
\end{cor}
\begin{proof}
Let $\mathcal{F}$ be the universal sheaf for $M_{X,\beta}$. The insertion (\ref{tau gamma})
\begin{align*}
\tau \colon H^{4}(X,\mathbb{Z}) &\rightarrow H^{2}(M_{X,\beta},\mathbb{Z}), \\
\tau(\alpha) &=(\pi_{M})_{*}(\pi_{X}^{*}\alpha\cup [\mathcal{F}])
\end{align*}
satisfies $\tau(H\cdot Y)=(H\cdot\beta)\cdot[M_{Y,\beta}]$ for any divisor $H\subseteq X$. By Lemma \ref{DT4/DT3}, then
\begin{equation}\mathrm{DT}_{4}(\beta\mid H\cdot Y)=\int_{[M_{X,\beta}]^{\rm{vir}}}\tau(H\cdot Y)=
\pm\deg[M_{Y,\beta}]^{\rm{vir}}\cdot(H\cdot\beta).  \nonumber \end{equation}
Hence, Conjecture \ref{conj:GW/GV} holds for $\beta\in H_{2}(Y)\subseteq H_{2}(X)$ and $\gamma=H\cdot Y$ if and only
\begin{equation}\label{equ3}\mathrm{GW}_{0, \beta}(H\cdot Y)=
\sum_{d|\beta}\frac{1}{d^{2}}\cdot\mathrm{DT}_{4}(\beta/d \mid H\cdot Y)=
\sum_{d|\beta}\frac{1}{d^{3}}\cdot\mathrm{DT}_{3}(\beta/d)\cdot(H\cdot\beta). \end{equation}
Notice that GW invariants satisfy similar formula as (\ref{DT4/DT3 invs}), i.e.
\begin{equation}\mathrm{GW}_{0,\beta}(H\cdot Y)=\mathrm{GW}_{0,\beta}(Y)\cdot(H\cdot\beta).  \nonumber \end{equation}
So (\ref{equ3}) holds true if and only if
\begin{equation}\mathrm{GW}_{0,\beta}(Y)=\sum_{d|\beta}\frac{1}{d^{3}}\cdot\mathrm{DT}_{3}(\beta/d), \nonumber \end{equation}
i.e. genus zero GV/DT conjecture (Conjecture~\ref{conj:katz})
 for $\beta\in H_{2}(Y)$ in a CY 3-fold $Y$.
\end{proof}
By combining with Corollary~\ref{cor:CI}, we have the following: 
\begin{thm}\label{DT4/DT3 primitive}
Let $Y$ be a complete intersection CY 3-fold in the 
product of projective 
spaces $\mathbb{P}^{n_1} \times \cdots \times \mathbb{P}^{n_k}$, 
and $X=Y \times E$ for an elliptic curve $E$. 
Then for any primitive curve class $\beta$ on $Y$
and cycle class $H \cdot Y$ for any divisor $H$ on $X$, 
Conjecture~\ref{conj:GW/GV} holds, i.e. the identity
\begin{equation}\mathrm{GW}_{0, \beta}(H \cdot Y)=\DT_4(\beta \mid H \cdot Y) \nonumber \end{equation}
holds for any primitive curve class $\beta\in H_{2}(Y)\subseteq H_{2}(X)$, any divisor $H\subseteq X$ and certain choice of orientation in defining the RHS.
\end{thm}

\subsection{Hyperk\"{a}hler 4-folds and cosection localization}
In this subsection, we investigate Conjecture \ref{conj:GW/GV}
 for a CY 4-fold $X$ which admits a holomorphic symplectic form, i.e. a hyperk\"{a}hler 4-fold. GW invariants on hyperk\"{a}hler manifolds vanish as they are deformation invariants and there are no holomorphic curves
for generic complex structures in the $\mathbb{S}^{2}$-twistor family.  An alternate way to see this vanishing is through the existence of
a nowhere-vanishing cosection (see for example Kiem-Li \cite{KL}).

Given a perfect obstruction theory \cite{BF, LT} on a Deligne-Mumford stack $M$, the existence of a cosection
\begin{equation}\varphi: Ob_{M}\rightarrow \mathcal{O}_M   \nonumber \end{equation}
of the obstruction sheaf $Ob_{M}$ allows us to localize the virtual class of $M$ to the closed subspace $Z(\varphi)\subseteq M$ where $\varphi$ is not surjective.
In particular, if $\varphi$ is surjective everywhere (which is guaranteed by the existence of holomorphic symplectic forms), then the virtual class of $M$ vanishes. Moreover, by truncating the obstruction theory to remove the trivial factor $\mathcal{O}_M$, one can define a reduced obstruction theory and reduced virtual class.

To verify Conjecture \ref{conj:GW/GV}
for hyperk\"{a}hler 4-folds, we need to prove the vanishing of $\mathrm{DT_4}$ invariants for $M_\beta$. Heuristically, in the ideal case, when all curves in $X$ are smooth embedded, one could identify the
obstruction theory of $M_\beta$ with obstruction theory of GW theory as in ~\cite[Section 7.2]{CL}, hence vanishing of invariants
follows. We give a cosection argument as follows.

${}$ \\
\textbf{Cosection and vanishing of $\mathrm{DT_4}$ virtual classes}.
Fix a 1-dimensional stable sheaf $F\in M_\beta$.  By taking the wedge product with the square $\mathrm{At}(F)^{2}$ of the Atiyah class and contracting with the holomorphic symplectic form $\sigma$, we get a surjective map
\begin{equation}\xymatrix@1{
\phi:\Ext^{2}(F,F)\ar[r]^{\wedge\frac{\mathrm{At}(F)^{2}}{2}}& \Ext^{4}(F,F\otimes \Omega^{2}_X) \ar[r]^{\quad \lrcorner \sigma} & \Ext^{4}(F,F)
\ar[r]^{tr} & H^{4}(X,\mathcal{O}_X)}.  \nonumber \end{equation}
More generally, we have
\begin{prop}\label{surj cosection}
Let $X$ be a projective hyperk\"{a}hler 4-fold, $F$ be a perfect complex on $X$ and $Q$ be the Serre duality quadratic form on $\Ext^{2}(F,F)$. Then the composition map
\begin{equation}\xymatrix@1{\phi:\Ext^{2}(F,F)\ar[r]^{\wedge\frac{\mathrm{At}(F)^{2}}{2}}& \Ext^{4}(F,F\otimes \Omega^{2}_X) \ar[r]^{\quad \lrcorner \sigma} & \Ext^{4}(F,F)\ar[r]^{tr} & H^{4}(X,\mathcal{O}_X)   } \nonumber \end{equation}
is surjective if either $\ch_{3}(F)\neq0$ or $\ch_{4}(F)\neq0$.
Moreover,   
\begin{enumerate}
\item if $\ch_{4}(F)\neq0$, then we have a $Q$-orthogonal decomposition
\begin{equation}\Ext^{2}(F,F)=\Ker(\phi)\oplus\mathbb{C}\langle \mathrm{At}(F)^2\lrcorner\textrm{ } \sigma \rangle,   \nonumber \end{equation}
where $Q$ is non-degenerate on each subspace; 
\item if $\ch_{4}(F)=0$ and $\ch_{3}(F)\neq0$, then we have a $Q$-orthogonal decomposition
\begin{equation}\Ext^{2}(F,F)=\mathbb{C}\left\langle \mathrm{At}(F)^2\lrcorner\textrm{ } \sigma, \kappa_{X}\circ \mathrm{At}(F) \right\rangle \oplus
(\mathbb{C}\left\langle \mathrm{At}(F)^2\lrcorner\textrm{ } \sigma, \kappa_{X}\circ \mathrm{At}(F)\right\rangle)^{\perp},  \nonumber \end{equation}
where $Q$ is non-degenerate on each subspace. Here $\kappa_{X}$ is the Kodaira-Spencer class which is Serre dual to $\ch_3(F)$.
\end{enumerate}
\end{prop}
\begin{proof}
(1) If $\ch_{4}(F)\neq0$, $\mathrm{At}(F)^{4}\in \Ext^{4}(F,F\otimes K_{X})$ is a nonzero element since
\begin{equation}tr(\mathrm{At}(F)^{4})=4!\cdot \ch_{4}(F)\in H^{4,4}(X,\mathbb{Q}).  \nonumber \end{equation}
We define an inclusion
\begin{equation}\iota:H^{4}(X,\mathcal{O}_{X})\rightarrow \Ext^{2}(F,F),   \nonumber \end{equation}
\begin{equation}1\mapsto \frac{1}{12\ch_{4}(F)}\cdot (\mathrm{At}(F)^{2}\lrcorner \textrm{ }\sigma).  \nonumber \end{equation}
Then $\phi\circ\iota=Id$ and gives a splitting
\begin{equation}\Ext^{2}(F,F)=\Ker(\phi)\oplus\mathbb{C}\left\langle \mathrm{At}(F)^2\lrcorner\textrm{ } \sigma \right\rangle. \nonumber \end{equation}
Note that $Q(\Ker(\phi),\mathrm{At}(F)^{2}\lrcorner \textrm{ }\sigma)=0$ from the definition of $Q$ and $\Ker(\phi)$.

(2) If $\ch_{3}(F)\neq0$, we denote $\beta\in H_2(X)$ to be the Poincar\'{e} dual of $\ch_3(F)$.
By the non-degeneracy of $\sigma$,  one can choose a first order deformation $\kappa_{X}\in H^{1}(X,TX)$ of $X$ such that
\begin{equation}\int_{\beta} \kappa_{X}\lrcorner \textrm{ } \sigma=1. \nonumber \end{equation}
By ~\cite[Proposition 4.2]{BFlenner}, the obstruction class $\kappa_{X}\circ \mathrm{At}(F)\in \Ext^{2}(F,F)$ satisfies
\begin{equation}tr(\kappa_{X}\circ \mathrm{At}(F)\circ \mathrm{At}(F)^{2})=-2\kappa_{X}\lrcorner \textrm{ } \ch_{3}(F).  \nonumber \end{equation}
Then
\begin{align*}
\int_{X}\phi(\kappa_{X}\circ \mathrm{At}(F))\wedge \sigma^{2}&= -\int_{X}(\kappa_{X}\lrcorner \textrm{ } \ch_{3}(F))\wedge\sigma  \\
&= \int_{X}(\kappa_{X}\lrcorner \textrm{ }\sigma )\wedge \ch_{3}(F)  \\
&= \int_{\beta}\kappa_{X}\lrcorner \textrm{ }\sigma =1,
\end{align*}
where the second equality is because of the homotopy formula ~\cite[Proposition 10]{MPT}
\begin{equation}0=\kappa_{X}\lrcorner \textrm{ }(\ch_{3}(F)\wedge\sigma)=(\kappa_{X}\lrcorner \textrm{ }\ch_{3}(F))\wedge\sigma+(\kappa_{X}\lrcorner \textrm{ }\sigma)\wedge \ch_{3}(F).  \nonumber \end{equation}
Thus the map
\begin{align*}
\iota:H^{4}(X,\mathcal{O}_{X}) &\rightarrow \Ext^{2}(F,F),   \\
1 & \mapsto \kappa_{X}\circ \mathrm{At}(F)   
\end{align*}
satisfies $\phi\circ\iota=\id$ and hence $\phi$ is surjective.

Notice that
\begin{align*}
Q(\kappa_{X}\circ \mathrm{At}(F),\mathrm{At}(F)^{2}\lrcorner \textrm{ } \sigma)
&=2\int_{X}\phi(\kappa_{X}\circ \mathrm{At}(F))\wedge\sigma^{2}=2, \\
\int_{X}\phi(\mathrm{At}(F)^{2}\lrcorner\textrm{ }\sigma)\wedge\sigma^{2}
&=\frac{1}{2}Q(\mathrm{At}(F)^{2}\lrcorner \textrm{ }\sigma,\mathrm{At}(F)^{2}\lrcorner\textrm{ }\sigma)=0, 
\end{align*}
since $\ch_4(F)=0$.  Thus  $\mathbb{C}\left\langle \mathrm{At}(F)^2\lrcorner\textrm{ } \sigma, \kappa_{X}\circ \mathrm{At}(F) \right\rangle$ is a two dimensional subspace on which $Q$ is non-degenerate.
The orthogonal complement $(\mathbb{C}\left\langle \mathrm{At}(F)^2\lrcorner\textrm{ } \sigma, \kappa_{X}\circ \mathrm{At}(F) \right\rangle)^{\perp}$ does not contain
$\mathrm{At}(F)^2\lrcorner\textrm{ } \sigma$ and $\kappa_{X}\circ \mathrm{At}(F)$, so we have
\begin{equation}\mathbb{C} \left\langle \mathrm{At}(F)^2\lrcorner\textrm{ } \sigma, \kappa_{X}\circ \mathrm{At}(F) \right\rangle\oplus(\mathbb{C}\langle \mathrm{At}(F)^2\lrcorner\textrm{ } \sigma, \kappa_{X}\circ \mathrm{At}(F)\rangle)^{\perp}=\Ext^{2}(F,F) \nonumber \end{equation}
by dimensions counting.
\end{proof}
We claim that the surjectivity of the cosection map leads to the vanishing of virtual class.
\begin{claim}\label{vanishing for hk4}
Let $X$ be a projective hyperk\"{a}hler 4-fold and $M$ be a proper moduli scheme of simple perfect complexes $F$'s with 
$\ch_{4}(F)\neq 0$ or $\ch_{3}(F)\neq 0$. Then the virtual class satisfies
\begin{equation}[M]^{\rm{vir}}=0.  \nonumber \end{equation}
\end{claim}
At the moment, we are lack of Kiem-Li type theory of cosection localization for D-manifolds in the sense of Joyce or Kuranishi space structures in
the sense of Fukaya-Oh-Ohta-Ono. We believe that when such a theory is established, our claim should follow automatically. 
Nevertheless, we have the following evidence for the claim.

1. At least when $M_\beta$ is smooth, Proposition \ref{surj cosection} gives the vanishing of virtual class.

2. If there is a complex analytic version of $(-2)$-shifted symplectic geometry \cite{PTVV} and the corresponding construction of virtual classes \cite{BJ},
one could prove the vanishing result as in $\mathrm{GW}$ theory, i.e. taking a generic complex structure in the $\mathbb{S}^{2}$-twistor family 
of the hyperk\"ahler 4-fold which does not support coherent sheaves and then vanishing of virtual classes follows from their deformation invariance.

${}$ \\
\textbf{Reduced $\mathrm{DT_4}$ virtual classes, an example}. By taking away the trivial factors in obstruction spaces, one could define reduced invariants, which
are computed in the following example.

Let $p: S\rightarrow\mathbb{P}^{1}$ be an elliptic $K3$ surface
with a section $i$. We assume general fibers of $p$ are smooth elliptic curves and any singular fiber is either a nodal or cuspidal plane curve.

Fix a CY surface $T$, we denote
\begin{align*}\pi \colon X =S\times T &\rightarrow \mathbb{P}^{1}\times T, \\
\pi(s,t)&=(p(s),t), 
\end{align*}
which is an elliptic fibration with a section $s=(i, \id)$.
Let $[f]$ be the fiber class of fibration $\pi$, and $\beta=r[f]\in H_2(X)$ with $r\geqslant1$.
As in 
Lemma~\ref{lem:fib:isom}, 
there exists an isomorphism
$M_{\beta}\cong X$
such that the $\mathrm{DT_4}$ virtual class satisfies 
\begin{align*}
[M_{\beta}]^{\rm{vir}}=\pm \mathrm{PD}(c_{3}(X))=0.
\end{align*}
Under the above isomorphism
$M_{\beta} \cong X$, the obstruction bundle of $M_{\beta}$ is
\begin{align*}
\wedge^{2}(TX)\cong \mathcal{O}_S\oplus\mathcal{O}_T\oplus (TS\otimes TT),  
\end{align*}
and the $\mathrm{DT_{4}}$ obstruction bundle can be chosen as
\begin{align*}\wedge^{2}_{+}(TX)=\mathcal{O}_S\oplus(TS\otimes TT)_{+}.  
\end{align*}
The trivial factor $\mathcal{O}_S$ in $\wedge^{2}_{+}(TX)$ makes the $\mathrm{DT_4}$ virtual class vanish.
We consider reduced obstruction bundle
\begin{equation}\wedge^{2}_{+}(TX)_{\rm{red}}:=(TS\otimes TT)_{+},  \nonumber \end{equation}
and reduced $\mathrm{DT_4}$ virtual class
\begin{equation}[M_{\beta}]_{\rm{red}}^{\rm{vir}}:=\mathrm{PD}(e(\wedge^{2}_{+}(TX)_{\rm{red}})). \nonumber \end{equation}
By the property of half Euler class (e.g. ~\cite[Remark 8.3]{CL} ), we have 
\begin{equation}e((TS\otimes TT)_{+})^{2}=e(TS\otimes TT)=-2c_{2}(S)\cdot c_{2}(T),  \nonumber \end{equation}
whose square roots are given by
\begin{equation}e(\wedge^{2}_{+}(TX)_{\rm{red}})=\pm\sqrt{e(TS\otimes TT)}=\pm\left(c_{2}(S)-c_{2}(T)\right).  \nonumber \end{equation}
Hence
\begin{equation}[M_{\beta}]_{\rm{red}}^{\rm{vir}}=\pm\left(\chi(S)\cdot[T]-\chi(T)\cdot[S] \right).  \nonumber \end{equation}
As for insertions, we consider
\begin{equation}\tau: H^{6}(X)\rightarrow H^{4}(M_{\beta}), \quad \tau(\gamma)=(\pi_{M_{\beta}})_{*}(\pi_{X}^{*}\gamma\cup \ch_{3}(\mathcal{E})),   \nonumber \end{equation}
where $\mathcal{E}$ is the universal sheaf and $\ch_3(\mathcal{E})=r[X\times_{(\mathbb{P}^{1}\times T)} X]$.

Then the \emph{reduced $\mathrm{DT_{4}}$ invariant}
\begin{equation}\mathrm{DT^{\rm{red}}_{4}}(r[f]\textrm{ }\big| \textrm{ }\gamma ):=\int_{[M_{\beta}]_{\rm{red}}^{\rm{vir}}}\tau(\gamma)
\nonumber \end{equation}
satisfies
\begin{equation}\mathrm{DT^{\rm{red}}_{4}}(r[f]\textrm{ }\big| \textrm{ }E)=r\int_{[M_{\beta}]_{\rm{red}}^{\rm{vir}}}\pi^{*}\pi_{*}(E)=\pm\, r\cdot\chi(S)\cdot\int_{T}[t]=\pm\, 24r, \nonumber \end{equation}
where $E=s(\mathbb{P}^{1}\times t)\in H_{2}(X)\cong H^{6}(X)$ as in Proposition \ref{prop on elliptic fib}.  \\

As for the corresponding GW theory, we have
\begin{equation}\overline{M}_{0,1}(X,r[f])\cong \overline{M}_{0,1}(S,r[f])\times T   \nonumber \end{equation}
whose virtual class vanishes. By considering the reduced obstruction theory \cite{KL} and insertions, the reduced GW invariant satisfies
\begin{align*}
\mathrm{GW}^{\rm{red}}_{0,r[f]}(E)&= \int_{[\overline{M}_{0,1}(X,r[f])]_{\rm{red}}^{\rm{vir}}}\mathrm{ev}^{*}(E) \\
&= \Big(\int_{[\overline{M}_{0,1}(S,r[f])]_{\rm{red}}^{\rm{vir}}}\mathrm{ev}^{*}(i(\mathbb{P}^{1}))\Big)\cdot\int_{T}[pt] \\
&= r\cdot\deg[\overline{M}_{0,0}(S,r[f])]_{\rm{red}}^{\rm{vir}},
\end{align*}
where $E=s(\mathbb{P}^{1}\times t)=i(\mathbb{P}^{1})\cdot 1 \in H_{2}(S)\otimes H_{0}(T)\subset H_{2}(X) $.  \\

A hyperk\"ahler version of Conjecture \ref{conj:GW/GV} for reduced invariants is given by
\begin{prop}\label{g=0 reduced inv conj}
In the above setting, we have a multiple cover formula  
\begin{equation}\mathrm{GW}^{\rm{red}}_{0,r[f]}(E)=\sum_{k|r}\frac{1}{k^{2}}\cdot\mathrm{DT^{\rm{red}}_{4}}\left(\frac{r}{k}[f]\textrm{ }\big| \textrm{ }E \right)
\nonumber \end{equation}
for certain choices of orientations in defining the RHS.
\end{prop}
\begin{proof}
We have the following Aspinwall-Morrison formula
\begin{equation}\deg[\overline{M}_{0,0}(S,r[f])]_{\rm{red}}^{\rm{vir}}=\sum_{k|r}\frac{1}{k^{3}}\cdot n_{0,\frac{r}{k}[f]} (S) \nonumber \end{equation}
relating reduced GW invariants $\deg[\overline{M}_{0,0}(S,r[f])]_{\rm{red}}^{\rm{vir}}$ with genus zero BPS
numbers $n_{0,r[f]}(S)$ for $K3$ surface $S$ (\cite{YZ, KMPS}). Yau-Zaslow formula gives
\begin{equation}n_{0,r[f]}(S)=n_{0,[f]}(S)=\chi(S)=24, \nonumber \end{equation}
as $r[f]\cdot r[f]=[f]\cdot[f]=0$.
\end{proof}

\section{Local surfaces}

Let $(S,\mathcal{O}_S(1))$ be a smooth projective surface and
\begin{align}\label{X:tot}
\pi \colon
X=\mathrm{Tot}_S(L_1 \oplus L_2) \to S
\end{align}
be the total space of direct sum of two line bundles $L_1$, $L_2$ on $S$.
Assuming that
\begin{align}\label{L12}
L_1 \otimes L_2 \cong K_S,
\end{align}
then $X$ is a non-compact CY 4-fold.
In this section, we study Conjecture~\ref{conj:GW/GV}
for $X$. 
\subsection{Stable sheaves without thickening}
For a non-compact CY 4-fold (\ref{X:tot}),
we take a curve class
\begin{align*}
\beta \in H_2(X, \mathbb{Z})\cong H_2(S, \mathbb{Z}),
\end{align*}
and consider the moduli space $M_{\beta}=M_{X,\beta}$ of 1-dimensional stable sheaves $F$ with $[F]=\beta$ and $\chi(F)=1$.  
We also consider $M_{S, \beta}$, the moduli space of 
1-dimensional stable sheaves $F$ on $S$ with $[F]=\beta$
and $\chi(F)=1$. 
Note that $M_{S, \beta}$ is compact while $M_{X, \beta}$ may not be compact in general. 
On the other hand, 
for the zero section 
$\iota \colon S \hookrightarrow X$ 
of the projection (\ref{X:tot}), we have the 
push-forward 
embedding
\begin{align}\label{tau:emb}
\iota_{\ast} \colon M_{S, \beta} \hookrightarrow M_{X, \beta}. 
\end{align}
In the following case, 
the morphism (\ref{tau:emb}) is an isomorphism and 
$M_{X, \beta}$ has well-defined $\mathrm{DT_{4}}$ virtual class.
\begin{prop}\label{vir:loc neg}
If $L^{-1}_{1}$ and $L^{-1}_{2}$ are ample, then
(\ref{tau:emb}) is an isomorphism. 
Under the isomorphism (\ref{tau:emb}), we have
\begin{align*}[M_{X, \beta}]^{\rm{vir}}=\pm [M_{S, \beta}]\cdot
e\left(\eE xt^1_{\pi_{M_S}}(\mathbb{F}, \mathbb{F} \boxtimes L_1)\right), \end{align*}
for certain choices of orientations in defining the LHS. 
Here $\mathbb{F} \in \Coh(S \times M_{S, \beta})$ is the universal sheaf and $\pi_{M_{S}}:S\times M_{S, \beta}\rightarrow M_{S, \beta}$ is the projection.
\end{prop}
\begin{proof}
A coherent sheaf $F\in \Coh(X)$ is determined by $\pi_{*}F\in \Coh(S)$ and two morphisms \cite[Ex. 5.17 Chapter II]{Hart}
\begin{equation}\phi_i:\pi_{*}F\rightarrow\pi_{*}F\otimes L_i,\textrm{ } i=1,2.    \nonumber \end{equation}
We claim $\phi_i=0$ by using the ampleness of $L_i^{-1}$.
Take the Harder-Narasimhan and Jordan-H\"{o}lder filtration 
\begin{equation}0=F_{0}\subseteq F_{1}\subseteq F_{2}\subseteq\cdot\cdot\cdot\subseteq F_{n}=\pi_*F  \nonumber \end{equation}
of $\pi_*F$,
where the quotient $E_{i}=F_{i}/F_{i-1}$'s are stable with decreasing reduced Hilbert polynomial.
\begin{equation}p(E_{1})\geqslant p(E_{2})\geqslant \cdot\cdot\cdot\geqslant p(E_{n}).  \nonumber \end{equation}
We consider the following diagram 
\begin{equation}
\xymatrix{
0 \ar[r] & F_1  \ar[r] & \pi_*F \ar[d]^{\phi_i} \ar[r] & \pi_*F/F_1 \ar[r] & 0 \\
0 \ar[r] & F_1\otimes L_i  \ar[r] & \pi_*F\otimes L_i \ar[r] &
(\pi_*F/F_1)\otimes L_i \ar[r] & 0. }
\nonumber \end{equation}
Note that $p(F_1)\geqslant p(E_k) > p(E_k\otimes L_i)$ by the ampleness of $L_i^{-1}$ for any $k=1,2,\cdot\cdot\cdot,n$,
hence $\Hom(F_1,E_k\otimes L_i)=0$ \cite[Proposition 1.2.7]{HL}. 

As $(\pi_*F/F_1)$ fits into extensions of $\{E_k\}_{k=2,\cdot\cdot\cdot,n}$, so 
\begin{equation}\Hom(F_1,(\pi_*F/F_1)\otimes L_i)=0. \nonumber \end{equation}
Hence $\phi_i$ 
restricts to $\phi_i|_{F_1}:F_1\to F_1\otimes L_i$. 
This determines a subsheaf $\widetilde{F_1}\subseteq F$ 
on $X$ such that $p(\pi_{*}\widetilde{F_1})=p(F_1)\geqslant p(\pi_*F)$, 
contradicting with the stability of $F$. 

Hence, for any $F\in M_{X, \beta}$, there exists $\mathcal{E}\in M_{S,\beta}$ such that $F=\iota_*(\mathcal{E})$.
To compare the deformation-obstruction theory, we have canonical isomorphisms
\begin{align*}
&\Ext^{1}_{X}(\iota_*\mathcal{E},\iota_*\mathcal{E})\cong \Ext^{1}_{S}(\mathcal{E},\mathcal{E}),  \\
&\Ext^{2}_{X}(\iota_*\mathcal{E},\iota_*\mathcal{E})\cong \Ext^{1}_{S}(\mathcal{E},\mathcal{E}\otimes L_{1})\oplus \Ext^{1}_{S}(\mathcal{E},\mathcal{E}\otimes L_{1})^{\vee},  
\end{align*}
where $\Ext^{2}_{S}(\mathcal{E},\mathcal{E})\cong \Ext^{0}_{S}(\mathcal{E},\mathcal{E}\otimes K_S)^\vee=0$ and $\Ext^{0}_{S}(\mathcal{E},\mathcal{E}\otimes L_{i})=0$ ($i=1,2$) by the stability of $\eE$.
Hence 
the morphism (\ref{tau:emb}) is an isomorphism. 
The comparison of virtual classes is similar to \cite[Theorem 6.5]{CL}
(see the last part of Section~\ref{sec:review}). 
\end{proof}

\subsection{Computations for $\mathcal{O}_{\mathbb{P}^{2}}(-1,-2)$}
\label{subsec:locP2}
In this subsection, we fix $S=\mathbb{P}^{2}$ and
consider the case 
 $L_{1}=\mathcal{O}(-1)$
and $L_2=\mathcal{O}(-2)$. 
 For $d\in\mathbb{Z}$, we 
consider the moduli space $M_{X, d}$ of one dimensional stable sheaves $F$'s on $X=\mathcal{O}_{\mathbb{P}^{2}}(-1,-2)$ with
\begin{align*}
[F]=d\in H_2(X, \mathbb{Z})\cong H_2(\mathbb{P}^2, \mathbb{Z})
\cong\mathbb{Z}, \ \chi(F)=1. 
\end{align*}
 By Lemma \ref{vir:loc neg}, we have
the isomorphism $M_{S, d} \stackrel{\cong}{\to}M_{X, d}$ and 
 a well-defined virtual class on
 $M_{X, d}$, whose computation is reduced to the 
one on the moduli space $M_{S,d}$
on $S$.
We explain the calculation in degree $d=3$, since degrees $1$ and $2$ are easier versions of the same approach.

There is a natural support morphism to the linear system of degree $3$ curves
$$M_{S, 3} \rightarrow |\mathcal{O}(3)| = \mathbb{P}^9.$$  
Moreover, if we denote 
$$\mathcal{C} \hookrightarrow \mathbb{P}^9 \times \mathbb{P}^2$$
the universal curve over this linear system,
we have an isomorphism
$\mathcal{C} \cong M_{S, 3}$
which sends the pair $(C,p)$ to the dual (on $C$) of the ideal sheaf $I_{C,p}$.

Let $\mathcal{V}$ denote the vector bundle on $M_{S, 3}$ whose fiber at a point $[E]$ is 
$$\Ext^1_S(E,E(-1)) = \RHom_S(I_{C,p},I_{C,p}(-1))[1].$$  
Its top Chern class can be computed via the diagram:
\begin{align*}
\xymatrix{
\cC \times \mathbb{P}^2
\ar@<-0.3ex>@{^{(}->}[r]^j
\ar[d]_{\pi_{\cC}}\ar@{}[dr]|\square &
 \mathbb{P}^9 \times \mathbb{P}^2 \times \mathbb{P}^2
 \ar[d]^{\pi_{1, 2}} \\
\cC \ar@<-0.3ex>@{^{(}->}[r]^{j} & \mathbb{P}^9 \times \mathbb{P}^2. 
}
\end{align*}
The $K$-theory class of the universal ideal sheaf
$\iI$ on $\mathcal{C}\times\mathbb{P}^2$ is given by the pullback
$[\mathcal{I}] = j^*[\fF]$, 
where $[\fF]$ is given by
\begin{align*}
[\fF] &= \pi_{1,3}^{*}[\mathcal{O}_{\mathcal{C}}] - \pi_{2,3}^{*}[\mathcal{O}_{\Delta}] \\
&= 3[\mathcal{O}(0,-1,0)] - [\mathcal{O}(-1,0,-3)] - [\mathcal{O}(0,-1,1)] - [\mathcal{O}(0,-2,-1)].
\end{align*}
Therefore
\begin{align*}
[\mathcal{V}] = j^*\pi_{1,2,*}([\fF]^{\vee}\otimes[\fF]\otimes\mathcal{O}(0,0,-1))
\end{align*}
is the pullback of an explicit $K$-theory class $\gamma$ on $\mathbb{P}^9\times\mathbb{P}^2$.
Note that we have 
 $$j_*[\mathcal{C}] = H_1+3H_2 \in H^2(\mathbb{P}^9\times\mathbb{P}^2,\mathbb{Z}),$$
where $H_1$, $H_2$ are hyperplane classes on $\mathbb{P}^9$,  
$\mathbb{P}^2$. 
For the point class
$[\mathrm{pt}] \in H^4(\mathbb{P}^2, \mathbb{Z})$, 
we can compute 
\begin{align*}
\int_{[M_{X, 3}]^{\rm{vir}}}
\tau([\rm{pt}])
&=\pm\int_{[M_{S, 3}]} e(\vV) \cdot j^{\ast}H_1 \\
&=\pm\int_{[\mathbb{P}^9\times\mathbb{P}^2]} 
(H_1+3H_2)\cdot H_1\cdot \gamma = \pm 1,
\end{align*}
which matches the prediction via Gromov-Witten theory in~\cite[Section~3.2]{KP}.

The same approach works for curves of degree $(2,2)$ on $\mathbb{P}^1 \times \mathbb{P}^1$ and this again matches the
answer via Gromov-Witten theory in~\cite[Section~3.3]{KP}.

\subsection{Localization principle in $\mathrm{DT_4}$ theory}
In general, $M_{X, \beta}$ may not be compact, so we want to define
the integral of $[M_{X, \beta}]^{\rm{vir}}$ via virtual localization:

Let $\mathbb{C}^{\ast}$ act on the fibers of
(\ref{X:tot}) by weight $(1, -1)$, which preserves the CY4-form on $X$. So the action also lifts to the moduli space $M_{X, \beta}$ preserving
the Serre duality pairing.
Analogous to \cite[Section 8]{CL}, heuristically speaking, one should have virtual localization formula of type
\begin{align}\label{vir:loc}
[M_{X, \beta}]^{\rm{vir}}=
[M_{X, \beta}^{\mathbb{C}^{\ast}}]^{\rm{vir}}
\cdot  e( \dR \hH om_{\pi_M}(\eE, \eE)^{\rm{mov}})^{1/2}
\in H_{\ast}(M_{X, \beta}^{\mathbb{C}^{\ast}})[t^{\pm 1}],
\end{align}
where $M_{X, \beta}^{\mathbb{C}^{\ast}}$ should have $(-2)$-shifted symplectic structure and $[M_{X, \beta}^{\mathbb{C}^{\ast}}]^{\rm{vir}}$ is its $\mathrm{DT_{4}}$ virtual class, $\eE\in\Coh(X\times M_{X, \beta})$ is 
the universal sheaf and $\pi_M:X\times M_{X, \beta}\to 
M_{X, \beta}$ is the projection, 
$t$ is the equivariant parameter for the $\mathbb{C}^{\ast}$-action.
\begin{rmk}
Suppose that $M_{X, \beta}$ admits a $\mathbb{C}^{\ast}$-equivariant virtual class
induced by the $\mathbb{C}^{\ast}$-equivariant $(-2)$-shifted symplectic structure.
Then the RHS of (\ref{vir:loc}) may coincide with the integration of the $\mathbb{C}^{\ast}$-equivariant virtual class up to sign
by a virtual $\mathbb{C}^{\ast}$-localization formula. 
\end{rmk}

\subsection{Contribution from surface component $M_{S, \beta}$}
Notice that the moduli space $M_{S, \beta}$ of one dimensional stable sheaves $F$'s
on $S$ with $[F]=\beta$, $\chi(F)=1$ is a union of connected components of
$M_{X, \beta}^{\mathbb{C}^{\ast}}$. Below we
 determine the contribution of the \emph{surface component} $M_{S, \beta}$ to the equivariant localization formula (\ref{vir:loc}) of $M_{X, \beta}$.

Let $\mathbb{F} \in \Coh(S \times M_{S, \beta})$ be the universal sheaf, $\pi_{M_{S}} \colon M_{S, \beta} \times S \to M_{S, \beta}$ be the projection.
The standard deformation-obstruction theory
\begin{align}\label{obs:U}
(\tau_{\geqslant 1}\dR \hH om_{\pi_{M_{S}}}(\mathbb{F}, \mathbb{F}))^{\vee}[-1] \to
\mathbb{L}_{M_{S, \beta}}
\end{align}
of $M_{S, \beta}$ is perfect \cite{BF, LT} and defines a virtual class
\begin{align*}
[M_{S, \beta}]^{\rm{vir}} \in H_{2\beta^2+2}(M_{S, \beta},\mathbb{Z}).
\end{align*}
\begin{prop-defi}\label{equiv vir cycle surface cpn}
Suppose that $\Hom(F, F \otimes L_2)=0$ for
any $[F] \in M_{S, \beta}$. The contribution of $M_{S, \beta}$ to the equivariant virtual class (\ref{vir:loc}) of $M_{X, \beta}$ is
\begin{align}\label{id:MSvir}
\pm
[M_{S, \beta}]^{\rm{vir}} \cdot
e(-\dR \hH om_{\pi_{M_{S}}}(\mathbb{F}, \mathbb{F} \boxtimes L_1) \otimes t).
\end{align}
\end{prop-defi}
\begin{proof}
Let $j$ be the inclusion
\begin{align*}
j=(\iota, \id) \colon S \times M_{S, \beta} \hookrightarrow X \times M_{S, \beta},
\end{align*}
where $\iota$ is the zero section inclusion.
We set
\begin{align*}
\uU \cneq \dR \hH om_{\pi_{M_{S}}}(j_{\ast}\mathbb{F}, j_{\ast}\mathbb{F})
\in D^{b}_{\mathbb{C}^{\ast}}(M_{S, \beta}).
\end{align*}
Then there are isomorphisms
\begin{align*}
\uU &\cong
\dR \hH om_{\pi_{M_{S}}}(\dL j^{\ast} j_{\ast}\mathbb{F}, \mathbb{F}) \\
& \cong \dR \hH om_{\pi_{M_{S}}}\Big(\mathbb{F}
 \oplus (\mathbb{F} \boxtimes
N_{S/X}^{\vee})[1] \oplus (\mathbb{F} \boxtimes \wedge^2 N_{S/X}^{\vee})[2],
\mathbb{F}\Big) \\
& \cong \dR \hH om_{\pi_{M_{S}}}(\mathbb{F}, \mathbb{F}) \oplus
\dR \hH om_{\pi_{M_{S}}}(\mathbb{F}, \mathbb{F} \boxtimes K_S)[-2] \\
& \qquad \oplus
\dR \hH om_{\pi_{M_{S}}}(\mathbb{F}, \mathbb{F} \boxtimes L_1) \otimes t [-1]
\oplus
\dR \hH om_{\pi_{M_{S}}}(\mathbb{F}, \mathbb{F} \boxtimes L_2) \otimes t^{-1} [-1].
\end{align*}
For the $\mathbb{C}^{*}$-fixed part, the Grothendieck duality gives
\begin{align*}
&\hH^1(\uU)^{\rm{fix}} \cong \eE xt^1_{\pi_{M_{S}}}(\mathbb{F}, \mathbb{F}), \\
&\hH^2(\uU)^{\rm{fix}} \cong \eE xt^2_{\pi_{M_{S}}}(\mathbb{F}, \mathbb{F})
\oplus \eE xt_{\pi_{M_{S}}}^2(\mathbb{F}, \mathbb{F})^{\vee}.
\end{align*}
By the shifted cotangent bundle argument as in 
Subsection~\ref{sec:review}, we see that
\begin{align*}
[M_{X, \beta}]^{\rm{vir}}|_{M_{S, \beta}}=[M_{S,\beta}]^{\rm{vir}}.
\end{align*}
For the movable part,
using condition (\ref{L12}) and
Grothendieck duality, we have
\begin{align*}
&\hH^1(\uU)^{\rm{mov}} \cong
\left(\hH om_{\pi_{M_{S}}}(\mathbb{F}, \mathbb{F} \boxtimes L_1) \otimes t \right)
\oplus \left(\hH om_{\pi_{M_{S}}}(\mathbb{F}, \mathbb{F} \boxtimes L_2) \otimes t^{-1}
\right), \\
&\hH^2(\uU)^{\rm{mov}} \cong \left(\eE xt^1_{\pi_{M_{S}}}(\mathbb{F},
\mathbb{F} \boxtimes L_1) \otimes t\right) \oplus \left(\eE xt^1_{\pi_{M_{S}}}(\mathbb{F},
\mathbb{F} \boxtimes L_1)^{\vee} \otimes t^{-1}\right).
\end{align*}
By the assumption, we have
$\hH om_{\pi_{M_{S}}}(\mathbb{F}, \mathbb{F} \otimes L_2)=0$
and
\begin{align*}
\eE xt^2_{\pi_{M_{S}}}(\mathbb{F}, \mathbb{F} \otimes L_1)=
\hH om_{\pi_{M_{S}}}(\mathbb{F}, \mathbb{F} \otimes L_2)^{\vee}=0.
\end{align*}
Therefore we obtain the desired identity (\ref{id:MSvir}) following localization formula of type (\ref{vir:loc}).
\end{proof}

\begin{defi}\label{defi:DT4res}
Suppose that $\Hom(F, F \otimes L_2)=0$ for any $[F] \in M_{S, \beta}$, 
and $M_{X, \beta}^{\mathbb{C}^{\ast}}=M_{S, \beta}$ holds. 
We define the residue $\mathrm{DT_4}$
 invariant by 
the residue of (\ref{equiv vir cycle surface cpn})
at $t=0$, i.e. 
\begin{align*}
\DT_4^{\rm{res}}(\beta) \cneq \pm\int_{[M_{S, \beta}]^{\rm{vir}}}
c_{\beta^2+1}(-\dR \hH om_{\pi_M}(\mathbb{F}, \mathbb{F} \boxtimes L_1)). 
\end{align*}
\end{defi}

In the GW side, 
let $\pi \colon \cC \to \overline{M}_{0, 0}(S, \beta)$ be the universal 
curve and $f \colon \cC \to S$ be the universal
stable map. 
We define the residue GW invariant by 
\begin{align}\label{res:GW}
\mathrm{GW}_{0, \beta}^{\rm{res}}
\cneq \mathrm{Res}_{t=0}
\int_{[\overline{M}_{0, 0}(S, \beta)]^{\rm{vir}}} e(- \dR \pi_{\ast}f^{\ast}N_{S/X}) \in \mathbb{Q}. 
\end{align}
Here $N_{S/X}=(L_1 \otimes t) \oplus (L_2 \otimes t^{-1})$ is the 
$\mathbb{C}^{\ast}$-equivariant normal bundle
and $e(-)$ is the $\mathbb{C}^{\ast}$-equivariant Euler class. 
As an analogy of Conjecture~\ref{conj:GW/GV}, we propose the following conjecture:
\begin{conj}\label{conj:red}
We have the following identity
\begin{align}\label{id:res}\mathrm{GW}_{0, \beta}^{\rm{res}}=\sum_{k|\beta} \frac{1}{k^3} \DT_4^{\rm{res}}(\beta/k),\end{align}
for certain choices of orientations in defining the RHS.
\end{conj}
Note that the power of $1/k$ in the coefficients is $3$ instead of $2$ because we do not put insertions here.

\subsection{Computations for $X=\mathrm{Tot}_S(\mathcal{O}_{S}\oplus K_{S})$}
Let $S$ be a smooth projective surface and 
consider the non-compact CY 4-fold
$X=\mathrm{Tot}_S(\mathcal{O}_{S}\oplus K_{S})$. 
In this case, the residue $\DT_4$ invariant 
in Definition~\ref{defi:DT4res}
is 
related to the $\DT_3$ invariant 
on the non-compact CY 3-fold 
$Y=\mathrm{Tot}_S(K_Y)$. 
We have the following lemma: 
\begin{lem}\label{lem:locDT4/3}
Suppose that $\Hom(F, F \otimes K_S)=0$ for any 
$[F] \in M_{S, \beta}$, and 
$M_{S, \beta}=M_{Y, \beta}$ holds. Then we have
\begin{align*}
\DT_4^{\rm{res}}(\beta)=\pm\chi(M_{S, \beta}).
\end{align*}
In particular, we have 
$\DT_4^{\rm{res}}(\beta)=\pm \DT_3(\beta)$, 
where the RHS is the $\DT_3$ invariant
on $Y$. 
\end{lem}
\begin{proof}
Since $X=Y \times \mathbb{A}^1$, 
similarly to Lemma~\ref{DT4/DT3}
we have $M_{X, \beta}=M_{Y, \beta} \times \mathbb{A}^1$. 
Therefore the assumption 
$M_{S, \beta}=M_{Y, \beta}$
implies $M_{X, \beta}^{\mathbb{C}^{\ast}}=M_{S, \beta}$. 
Also the assumption $\Hom(F, F \otimes K_S)=0$ implies
that 
$\Ext_S^2(F, F)=0$
for any $[F] \in M_{S, \beta}$.
In particular, 
$M_{S, \beta}$ is non-singular
and $[M_{S, \beta}]^{\rm{vir}}=[M_{S, \beta}]$.  
From the definition of 
$\DT_4^{\rm{res}}(\beta)$, it follows that 
\begin{align*}
\DT_4^{\rm{res}}(\beta)=\pm\int_{M_{S, \beta}}
e(\eE xt_{\pi_M}^1(\mathbb{F}, \mathbb{F}))
\end{align*}
which is the topological 
Euler characteristic of $M_{S, \beta}$. Therefore the lemma follows. 
\end{proof}
In the following examples, conditions in Definition \ref{defi:DT4res} are satisfied, so Lemma \ref{lem:locDT4/3} can be used to verify Conjecture \ref{conj:red}. 

${}$ \\
\textbf{When $S$ is a toric del-Pezzo surface}.
In the toric del-Pezzo surface case, we have 
\begin{thm}\label{toric del-Pezzo}
Let $S$ be a smooth  
toric del-Pezzo surface. 
Then Conjecture~\ref{conj:red} holds for 
$X=\mathrm{Tot}_S(\oO_S \oplus K_S)$. 
\end{thm}
\begin{proof}
By the virtual localization formula, 
it is easy to see that the residue GW invariant 
(\ref{res:GW}) is the usual GW invariant on $Y=\mathrm{Tot}_S(K_S)$. 
Then the result follows from 
Lemma~\ref{lem:locDT4/3} and Corollary~\ref{cor:toric}. 
\end{proof}


${}$ \\
\textbf{When $S$ is a rational elliptic surface}.
Let $p:S\rightarrow\mathbb{P}^{1}$ be a rational elliptic surface with a section $s$ and $C\triangleq s(\mathbb{P}^{1})\subseteq S$, $f$ be a general fiber of $p$. We consider primitive curve classes
\begin{align}\label{beta:n}
\beta_{n}=[C]+n[f]\in H_{2}(S,\mathbb{Z}),
\quad n\geqslant 0.
\end{align}
We first show the following: 
\begin{lem}\label{lem:ratell}
For any $[F] \in M_{S, \beta_n}$, we have 
$\Hom(F, F\otimes K_S)=0$ and
\begin{align*} 
M_{S,\beta_{n}} =M_{Y, \beta_n}= M_{X,\beta_{n}}^{\mathbb{C}^{*}}
\cong \Hilb^n(S).
\end{align*}
\end{lem}
\begin{proof}
By ~\cite[Proposition 4.8]{HST}, we have an isomorphism
$M_{S,\beta_{n}}=M_{Y, \beta_n} \cong \Hilb^{n}(S)$
given by a Fourier-Mukai transformation. 
In particular, the equivalence of derived categories implies that
for any $F\in M_{S,\beta_{n}}$, we have
\begin{align*}\Ext^{*}_S(F,F)\cong \Ext^{*}_S(I_{Z},I_{Z})
\end{align*}
for some 
zero dimensional subscheme $Z \subset S$ with length $n$. 
Hence 
\begin{align*}
\Hom(F,F\otimes K_S)\cong \Ext^{2}(F,F)^{\vee}
\cong \Ext_S^2(I_Z, I_Z)^{\vee} \cong \Hom(I_Z, I_Z \otimes K_S)=0
\end{align*}
 since $h^{0,i}(S)=0$ for $i=1, 2$. 
\end{proof}
As a corollary, we have the following: 
\begin{cor}\label{rational elliptic surface}
In the above situation, we have
$\DT_4^{\rm{res}}(\beta_n)=\pm \chi(\Hilb^n(S))$. 
In particular
if we take the plus sign as the orientation, 
they fit into the generating series
\begin{align*}
\sum_{n\geqslant 0}
\DT_4^{\rm{res}}(\beta_n)q^n
=\prod_{k\geqslant1}\frac{1}{(1-q^{k})^{12}}.
\end{align*}
\end{cor}
\begin{proof}
The identity $\DT_4^{\rm{res}}(\beta_n)=\pm \chi(\Hilb^n(S))$
follows from 
Lemma~\ref{lem:locDT4/3}
and Lemma~\ref{lem:ratell}. 
By G\"{o}ttsche's formula \cite{Gott}, we obtain the generating series.
\end{proof}
Now we have the following: 
\begin{thm}\label{rational elliptic}  
Let $S$ be a rational elliptic surface
and take $\beta_n$ as in (\ref{beta:n}). 
Then Conjecture~\ref{conj:red} 
is true for 
$X=\mathrm{Tot}_S(\oO_S \oplus K_S)$ and 
$\beta=\beta_n$, i.e. 
$\mathrm{GW}_{0, \beta_n}^{\rm{res}}=
\DT_4^{\rm{res}}(\beta_n)$ holds. 
\end{thm}
\begin{proof}
Since the virtual dimension of 
$\overline{M}_{0, 0}(S, \beta_n)$ is zero, we have 
\begin{align*}
\mathrm{GW}_{0, \beta_n}^{\rm{res}}=
\deg[\overline{M}_{0, 0}(S, \beta_n)]^{\rm{vir}}. 
\end{align*}
Its generating series 
fits into G\"{o}ttsche-Yau-Zaslow formula  (see e.g. ~\cite[Theorem 1.2]{BL})
\begin{align*}
\sum_{n=0}^{\infty}\textrm{deg}[\overline{M}_{0,0}(S,\beta_{n})]^{\rm{vir}}q^n=\prod_{k\geqslant1}\frac{1}{(1-q^{k})^{12}}. 
\end{align*}
Therefore the
result follows from Corollary~\ref{rational elliptic surface}. 
\end{proof}

\section{Local curves}
Let $C$ be a smooth projective curve of genus $g(C)=g$, and
\begin{align}\label{X:tot2}
p \colon
X=\mathrm{Tot}_C(L_1 \oplus L_2 \oplus L_3) \to C
\end{align}
be the total space of split rank three vector bundle on it.
Assuming that
\begin{align}\label{L123}
L_1 \otimes L_2 \otimes L_3 \cong \omega_C,
\end{align}
then the variety (\ref{X:tot2}) is a non-compact CY 4-fold.
Below we set $l_i \cneq \deg L_i$, and may assume that
$l_1 \geqslant l_2 \geqslant l_3$ without loss of generality.
In this section, we study an equivariant version of Conjecture~\ref{conj:GW/GV}
for $X$. 
\subsection{Localization for GW invariants}
Let $T=(\mathbb{C}^{\ast})^{\times 3}$ be the three dimensional complex torus
which acts on the fibers of $X$
given by (\ref{X:tot2}). 
Let $\bullet$ denote $\Spec \mathbb{C}$ with trivial $T$-action.
Let $\mathbb{C} \otimes t_i$ be the one dimensional vector
space with $T$-action
with weight $t_i$,
and
$\lambda_i \in H_T^{\ast}(\bullet)$
its 1st Chern class.
We note that
\begin{align}\label{H:lambda}
H_{T}^{\ast}(\bullet)=\mathbb{C}[\lambda_1, \lambda_2, \lambda_3]. 
\end{align}
Let $j \colon C \hookrightarrow X$
be the zero section of the projection (\ref{X:tot2}).
Note that we have
\begin{align*}
H_2(X, \mathbb{Z})=\mathbb{Z}[C],
\end{align*}
where $[C]$ is the fundamental class
of $j(C)$.
For $d \in \mathbb{Z}_{>0}$,
we consider the diagram
\begin{align*}
\xymatrix{
\cC  \ar[r]^{f} \ar[d]^{h} & C \\
\overline{M}_h(C, d[C]), &
}
\end{align*}
where $\cC$ is the universal curve and $f$ is the universal stable map.
The $T$-equivariant GW invariant
of $X$ is defined
by
\begin{align*}
\mathrm{GW}_{h, d[C]}=
\mathrm{GW}_{h, d} \cneq \int_{[\overline{M}_h(C, d[C])]^{\rm{vir}}}
e(-\dR h_{\ast}f^{\ast}N) \in \mathbb{Q}(\lambda_1, \lambda_2, \lambda_3),
\end{align*}
where
$N$ is the
$T$-equivariant normal bundle of $j(C) \subset X$:
\begin{equation}\label{nor bdl N}
 N=(L_1 \otimes t_1) \oplus (L_2 \otimes t_2) \oplus (L_3 \otimes t_3). 
\end{equation}

If $g(C)>0$, we have the obvious
vanishing of genus zero GW invariants
\begin{align*}
\mathrm{GW}_{0, d}=0, \ g(C)>0, \ d \in \mathbb{Z}_{>0}
\end{align*}
because $\overline{M}_0(C, d[C])=\emptyset$.

If $g(C)=0$, we have
\begin{align*}
\mathrm{GW}_{0, d}
=\int_{[\overline{M}_0(\mathbb{P}^1, d)]}
e\Big(-\dR h_{\ast}f^{\ast}\left(\oO_{\mathbb{P}^1}(l_1)t_1 \oplus
\oO_{\mathbb{P}^1}(l_2)t_2 \oplus
\oO_{\mathbb{P}^1}(l_3)t_3
\right)\Big).
\end{align*}
For example in the $d=1$ case,
$\overline{M}_0(\mathbb{P}^1, 1)$ is one point and
\begin{align}\label{GW01}
\mathrm{GW}_{0, 1}
=\lambda_1^{-l_1-1} \lambda_2^{-l_2-1} \lambda_3^{-l_3-1}.
\end{align}
In the $d=2$ case, a straightforward localization calculation
as in~\cite{Kont}
with respect to the $(\mathbb{C}^{\ast})^2$-action
on $\mathbb{P}^1$
gives
\begin{align}\label{GW02}
\mathrm{GW}_{0, 2}=&\frac{1}{8}
\lambda_1^{-2l_1-1}\lambda_2^{-2l_2-1}\lambda_3^{-2l_3-1}
\left\{(\overline{l}_1^2-(\overline{l}_1-1)^2+\cdots)
\lambda_1^{-2} + \right. \\
\notag &\left.
(\overline{l}_2^2-(\overline{l}_2-1)^2+\cdots) \lambda_2^{-2}
+(\overline{l}_3^2-(\overline{l}_3-1)^2+\cdots) \lambda_3^{-2} \right. \\
 \notag & \left.
+l_1 l_2 \lambda_1^{-1} \lambda_2^{-1}
+l_2 l_3 \lambda_2^{-1} \lambda_3^{-1}
+l_1 l_3 \lambda_1^{-1} \lambda_3^{-1} \right\}.
\end{align}
Here we write $\overline{l}=l$ for $l\geqslant 0$
and $\overline{l}=-l-1$ for $l<0$.

\subsection{Localization for stable sheaves on local curves}
The $T$-action on $X$ does not preserve its CY 4-form, 
so we cannot apply the $T$-localization for $\DT_4$-theory. 
Instead, we consider its restriction to the subtorus
\begin{align*}
T_0=\{t_1 t_2 t_3=1\} \subset T
\end{align*}
which preserves the CY 4-form on $X$. Then the Serre duality pairing on moduli space $M_\beta$ is preserved by $T_0$.
Similarly to the case of local surfaces, we will define equivariant virtual classes for $M_\beta$
using localization formula with respect to $T_0$-action, and investigate their relation with equivariant GW invariants.

We consider the $T_0$-action on 
the moduli space of one dimensional stable sheaves $M_{d[C]}$
on the local curve $X$. 
Let $\eE \in \Coh(X \times M_{d[C]}^{T_0})$ be the 
universal sheaf, and 
$\pi$ the projection from $X \times M_{d[C]}^{T_0}$
to $M_{d[C]}^{T_0}$.  
Analogous to (\ref{vir:loc}), heuristically speaking, virtual localization formula for $M_{d[C]}$ should give
the following definition: 
\begin{align}\label{DT4:locurve}
\DT_4(d[C])=\int_{[M_{d[C]}^{T_0}]^{\rm{vir}}}
e( \dR \hH om_{\pi}(\eE, \eE)^{\rm{mov}})^{1/2}
\in \mathbb{Q}(\lambda_1, \lambda_2),
\end{align}
where we use the isomorphism
\begin{align*}
H_{T_0}(\bullet)=\mathbb{C}[\lambda_1, \lambda_2, \lambda_3]/(\lambda_1+\lambda_2+\lambda_3)
\cong \mathbb{C}[\lambda_1, \lambda_2]. 
\end{align*}
This equivariant invariant (\ref{DT4:locurve}) will be rigorously defined for $d=1$ (\ref{DT4:C}) and $d=2$ (\ref{DT4-2C}) respectively below,
where the virtual class of the torus fixed locus is its usual fundamental class and there is a preferred choice of square root of the 
equivariant Euler class of the virtual normal bundle. 

\subsection{Contribution from the component $M_C(d, 1)$}
Let $M_C(d, 1)$ be the moduli space of rank $d$
stable vector bundles $F$ on $C$
with $\chi(F)=1$. 
By push-forward $j_{\ast}$ for the zero section $j \colon C \hookrightarrow X$, 
$M_C(d, 1)$ is a connected component of 
$M_{d[C]}^{T_0}$. Below we use the following fact:
\begin{lem}\label{lem:adjoint}
For $F, F' \in \Coh(C)$
and push-forwards $j_{\ast}F, j_{\ast}F' \in \Coh(X)$, 
 we have
\begin{align*}
\chi(j_{\ast}F, j_{\ast}F')
=\chi_C(F, F')-\chi_C(F, F' \otimes N)
+\chi_C(F, F' \otimes \wedge^2 N)-\chi_C(F, F' \otimes \wedge^3 N).
\end{align*}
Here $\chi(-, -)$ is the Euler pairing on $X$ and 
$\chi_C(-, -)$ is the Euler pairing on $C$. 
\end{lem}
\begin{proof}
For $F \in \Coh(C)$, we have the isomorphism
\begin{align*}
\dL j^{\ast}j_{\ast}F \cong \bigoplus_{i\geqslant 0} F \otimes \wedge^i N^{\vee}[i].
\end{align*}
Therefore the lemma follows by
the adjunction.
\end{proof}
We now investigate the contribution of $M_C(d, 1)$ to (\ref{DT4:locurve}). 
For $[F] \in M_C(d, 1)$,
by Lemma~\ref{lem:adjoint}
we have
\begin{align*}
\chi(j_{\ast}F, j_{\ast}F)
=\chi_C(F, F)-\chi_C(F, F \otimes N)+\chi_C(F, F \otimes \wedge^2 N)
-\chi_C(F, F \otimes \wedge^3 N).
\end{align*}
We set
\begin{align*}
\chi(j_{\ast}F, j_{\ast}F)^{1/2} \cneq \chi_C(F, F)-\chi_C(F, F \otimes N).
\end{align*}
Then as 
elements
of $T_0$-equivariant K-theory of $M_{C}(d, 1)$, 
we have
\begin{align*}
\chi(j_{\ast}F, j_{\ast}F)=\chi(j_{\ast}F, j_{\ast}F)^{1/2}+\chi(j_{\ast}F, j_{\ast}F)^{1/2, \vee}.
\end{align*}
The $T_0$-fixed and movable parts of $\chi(j_{\ast}F, j_{\ast}F)^{1/2}$ are given by
\begin{align}\label{T_0:mov}
\chi(j_{\ast}F, j_{\ast}F)^{1/2, \rm{fix}} &=\chi_C(F, F), \\
\notag
\chi(j_{\ast}F, j_{\ast}F)^{1/2, \rm{mov}} &=-\chi_C(F, F \otimes N).
\end{align}
The first identity of (\ref{T_0:mov}) implies that the 
virtual class $[M_{d[C]}^{T_0}]^{\rm{vir}}$ should be the usual 
fundamental class $[M_C(d, 1)]$ on the component 
$M_C(d, 1) \subset M_{d[C]}^{T_0}$. 
Let $\fF$ be a universal vector bundle on $C \times M_C(d, 1)$ and
$\pi$ the projection from
$M_C(d, 1) \times C$ to $M_C(d, 1)$.
By the second identity of (\ref{T_0:mov}),
the contribution of $M_C(d, 1)$
to (\ref{DT4:locurve}) should be given by 
\begin{align}\label{DT4:C}
 \int_{M_C(d, 1)}
e(-\dR \hH om_{\pi}(\fF, \fF \boxtimes N))
\in \mathbb{Q}(\lambda_1, \lambda_2),
\end{align}
where we regard $N$ (\ref{nor bdl N}) as a $T_0$-equivariant vector bundle on $C$.

\begin{prop}\label{cor on DT4 C}
The integral (\ref{DT4:C}) is zero 
unless $g(C)=0$ and 
$d=1$. 
If $g(C)=0$ and $d=1$, it is equal to 
\begin{align}\label{DT4(C)}
\lambda_1^{-l_1-1} \lambda_2^{-l_2-1}
(-\lambda_1-\lambda_2)^{-l_3-1}.
\end{align}
\end{prop}
\begin{proof}
We first assume that $g(C)=0$. 
Then we have 
 $M_C(d, 1) =\emptyset$ 
for $d\geqslant 2$, so (\ref{DT4:C}) is zero. 
For $g(C)=0$ and $d=1$, then 
$M_C(1, 1)=\Spec \mathbb{C}$ and the 
identity (\ref{DT4(C)})
  follows from
\begin{align*}
\dR \hH om_{\pi}(\fF, \fF \boxtimes N)
=\dR \Gamma(C, N). 
\end{align*}
Next we assume that $g(C)>0$ and 
show the vanishing of (\ref{DT4:C}). 
For $L \in \Pic(C)$, let
$M_C(d, L) \subset M_C(d, 1)$ be the closed subscheme
given by stable bundles $F$ on $C$ with
fixed determinant $\det F =L$.
We have an \'etale surjective map
\begin{align*}
h \colon
\Pic^0(C) \times M_C(d, L) \to M_C(d, 1)
\end{align*}
given by $(L', F) \mapsto L' \otimes F$.
By setting $\fF'=(h \times \id)^{\ast}\fF$, it is enough
to show the vanishing
\begin{align*}
\int_{\Pic^0(C) \times M_C(d, L)}e(-\dR \hH
om_{\pi'}(\fF', \fF' \boxtimes N))=0,
\end{align*}
where $\pi'$ is the projection from
$\Pic^0(C) \times M_C(d, L) \times C$ to $\Pic^0(C) \times M_C(d, L)$.

We set $\fF_L=\fF|_{M_C(d, L) \times C}$.
Since $\fF'$ is isomorphic to
$\oO_{\Pic^0(C)} \boxtimes
\fF_L$
up to tensoring a line bundle,
the above integral coincides with
\begin{align*}
\int_{\Pic^0(C) \times M_C(d, L)} p_M^{\ast}
e(-\dR \hH om_{\pi''}(\fF_L, \fF_L \boxtimes N)),
\end{align*}
where $p_M$ is the projection from $\Pic^0(C) \times M_C(d, L)$ to $M_C(d, L)$,
and $\pi''$ is the projection from $M_C(d, L) \times C$ to $M_C(d, L)$.
Since $\dim \Pic^0(C)=g(C)>0$, the above integral
obviously vanishes.
\end{proof}

When $d=1$, $M_C(1, 1)\cong M_{[C]}^{T_0}$ is the only $T_0$-fixed component, we have the following corollary: 
\begin{cor}\label{cor on deg one local curve}
For $g(C)>0$, we have 
$\mathrm{GW}_{0, [C]}=\DT_4([C])=0$. For $g(C)=0$, 
we have 
\begin{align*}
\mathrm{GW}_{0, [C]}|_{\lambda_3=-\lambda_1-\lambda_2}
=\DT_4([C])=
\lambda_1^{-l_1-1} \lambda_2^{-l_2-1} (-\lambda_1 -\lambda_2)^{-l_3-1}. 
\end{align*}
\end{cor}
\begin{proof}
For $d=1$, we have 
the isomorphism 
\begin{align*}
j_{\ast} \colon M_C(1, 1) =\Pic^{g}(C) \stackrel{\cong}{\to}
M_{[C]}^{T_0}. 
\end{align*}
Therefore the result follows from (\ref{GW01}) and 
Proposition~\ref{cor on DT4 C}. 
\end{proof}

\subsection{Localization for degree two stable sheaves}
We next
consider the case of $d=2$. 
For $(k, k') \in \mathbb{Z}^2$, we
denote by $\Pic^{(k, k')}(C)$
the moduli space of triples
\begin{align*}
(L, L', \iota), \ 
(L, L') \in \Pic^{k}(C) \times \Pic^{k'}(C), \ 
\iota \colon L \hookrightarrow L'
\end{align*}
Here $\iota$ is an inclusion of line bundles. 
Then $M_{2[C]}^{T_0}$ can be described as follows.
\begin{prop}\label{isom:deg2}
We have the following isomorphism
\begin{align}\label{isom:2[C]}
M_C(2, 1)\sqcup \coprod_{i=1}^3
\coprod_{\begin{subarray}{c}
(d_0, d_i) \in \mathbb{Z}^2 \\
d_0+d_i=2g-1+l_i \\
g\leqslant d_0 \leqslant g+(l_i-1)/2
\end{subarray}}
\Pic^{(d_0, d_i)}(C)
\stackrel{\cong}{\to} M^{T_0}_{2[C]}.
\end{align}
\end{prop}
\begin{proof}
For $[F] \in M_{2[C]}^{T_0}$,
it is either a rank two stable vector bundle on $C$ or
thickened into one of the $L_i$-direction.
In the latter case, we have
\begin{align*}
p_{\ast}F=F_0 \oplus (F_i \otimes t_i^{-1})
\end{align*}
where $F_0$, $F_i$ are line bundles on $C$. The
$\oO_X$-module structure of $F$ is
given by a non-trivial morphism
\begin{align}\label{F:phi}
\phi \colon F_0 \otimes L_i^{-1} \to F_i.
\end{align}
We write $F_i=F_i' \otimes L_i^{-1}$,
and set $d_0=\deg F_0$, $d_i=\deg F_i'$.
By the above argument,
the $L_i$-thickened sheaf $[F] \in M_{2[C]}^{T_0}$
determines a point
\begin{align*}
(F_0, F_i', \phi \colon F_0 \hookrightarrow F_i') \in
\Pic^{(d_0, d_i)}(C).
\end{align*}
Conversely by going back the above argument,
we easily see that any point in $\Pic^{(d_0, d_i)}(C)$
determines a point in $M_{2[C]}^{T_0}$.
Therefore
it is enough to determine the possible
$(d_0, d_i)$.

Since $\chi(F)=1$, we have the identity
\begin{align}\label{d_0+d_1}
d_0+d_i=2g-1+l_i.
\end{align}
By the exact sequence in $\Coh(X)$
\begin{align}\label{exact:FL}
0 \to F_i \to F \to F_0 \to 0
\end{align}
and the stability of $F$,
we have $\chi(F_0) \geqslant 1$, i.e.
$d_0 \geqslant g$.
Also since (\ref{F:phi}) is non-zero,
we have $d_0 \leqslant d_i$.
By substituting (\ref{d_0+d_1}),
we see that
\begin{align}\label{gd_0}
g \leqslant d_0 \leqslant g+\frac{l_i}{2}-\frac{1}{2}.
\end{align}
Conversely if $(d_0, d_i)$ satisfy the
conditions (\ref{d_0+d_1}), (\ref{gd_0}), then
the corresponding sheaf on $X$
determines a point in $M_{[2C]}^{T_0}$.
Therefore we obtain the desired isomorphism (\ref{isom:2[C]}).
\end{proof}
The following lemma is obvious: 
\begin{lem}\label{lem:Picd}
We have an isomorphism
\begin{align*}
\Pic^{(d, d')}(C) \stackrel{\cong}{\to} \Pic^{d'}(C) \times
S^{d'-d}(C)
\end{align*}
given by $(L, L', \iota) \mapsto (L', \Supp(\Cok(\iota)))$.
In particular, $\Pic^{(d, d')}(C)$ is smooth of
dimension $g+d'-d$.
\end{lem}
${}$ \\
\textbf{Determine $[M_{2[C]}^{T_0}]^{\rm{vir}}$ and square root of Euler class of virtual normal bundle}.
To apply localization formula (\ref{DT4:locurve}), we need to determine the virtual class $[M_{2[C]}^{T_0}]^{\rm{vir}}$, the equivariant virtual normal bundle of $M_{2[C]}^{T_0}$ and square roots of its Euler class.
By Proposition~\ref{cor on DT4 C}, 
the component $M_C(2,1)$
does not contribute to (\ref{DT4:locurve}). 
We need to consider contributions from 
thickened sheaves. 

For a $L_i$-thickened sheaf
$[F] \in M_{2[C]}^{T_0}$, the exact sequence
(\ref{exact:FL}) gives $F=F_0+F_i \cdot t_i^{-1}$ in the
$T_0$-equivariant $K$-theory of $X$.
Therefore
\begin{align*}
\chi(F, F)=\chi(j_{\ast}F_0, j_{\ast}F_0)+\chi(j_{\ast}F_0, j_{\ast}F_i) 
t_i^{-1} +
\chi(j_{\ast}F_i, j_{\ast}F_0)t_i + \chi(j_{\ast}F_i, j_{\ast}F_i).
\end{align*}
We set
\begin{align}\label{chi:half}
\chi(F, F)^{1/2} \cneq \chi(j_{\ast}F_0, j_{\ast}F_0)+\chi(j_{\ast}F_0, 
j_{\ast}F_i) t_i^{-1}.
\end{align}
Then as 
elements of $T_0$-equivariant $K$-theory of $M_{2[C]}^{T_0}$, 
we obtain
\begin{align*}
\chi(F, F)=\chi(F, F)^{1/2}+\chi(F, F)^{1/2, \vee}.
\end{align*}
By Lemma~\ref{lem:adjoint},
we have
\begin{align*}
\chi(F, F)^{1/2}&=\chi(\oO_C)-\chi(N)+\chi(\wedge^2 N)-\chi(\wedge^3 N) \\
&+(\chi(A)-\chi(A \otimes N)+\chi(A \otimes \wedge^2 N)-\chi(A \otimes
 \wedge^3 N))t_i^{-1},
\end{align*}
where we set $A=F_i \otimes F_0^{\vee}$.
By the above formula, the 
$T_0$-fixed part of (\ref{chi:half}) is
\begin{align*}
\chi(F, F)^{1/2, \rm{fix}}=
\chi(\oO_C)-\chi(\wedge^3 N)-\chi(A \otimes L_i)
\end{align*}
which is $(1-g-d_i+d_0)$-dimensional by Riemann-Roch theorem.
Therefore
\begin{align*}
\Hom(F, F)-\chi(F, F)^{1/2, \rm{fix}}
\end{align*}
is of dimension $g+d_i-d_0$, which
coincides with the dimension of
$\Pic^{(d_0, d_i)}(C)$ by Lemma~\ref{lem:Picd}.
It follows that, by Proposition~\ref{isom:deg2},
the virtual class associated to the
$T_0$-fixed obstruction theory
on $M_{2[C]}^{T_0}$ should be its usual fundamental class.  

We now give a definition of $\DT_4(2[C]) \in \mathbb{Q}(\lambda_1, \lambda_2)$.  
Let 
\begin{align*}
(\fF_0, \fF_i', \iota), \
\iota \colon \fF_0 \hookrightarrow \fF_i'
\end{align*}
be the universal object on $\Pic^{(d_0, d_i)}(C) \times C$, i.e.
$\fF_0, \fF_i'$ are
line bundles on $\Pic^{(d_0, d_i)}(C) \times C$
and $\iota$ is the universal injection.
Let $\fF_i \cneq \fF_i' \boxtimes L_i^{-1}$, 
and consider its push-forward
\begin{align*}
j_{\ast} \fF_i \in \Coh(\Pic^{(d_0, d_i)}(C) \times X), \ i=1, 2.
\end{align*}
Based on the localization formula (\ref{DT4:locurve}) and the above discussions, we define $\DT_{4}(2[C])$ to be
\begin{equation}\label{DT4-2C}
\sum_{i=1}^{3} \sum_{\begin{subarray}{c}(d_0, d_i) \in \mathbb{Z}^2 \\d_0+d_i=2g-1+l_i \\
g\leqslant d_0 \leqslant g+(l_i-1)/2 \end{subarray}}
\int_{\Pic^{(d_0, d_i)}(C)}e\big(\dR \hH om_{\pi}(j_{\ast}\fF_0, j_{\ast}\fF_0)^{\rm{mov}}+ 
\dR \hH om_{\pi}(j_{\ast}\fF_0, j_{\ast}\fF_i \cdot t_i^{-1})^{\rm{mov}} \big),   
\end{equation}
as an element in $\mathbb{Q}(\lambda_1, \lambda_2)$. 
Here $\pi$ is the projection
from $M_{2[C]}^{T_0} \times X$ to $M_{2[C]}^{T_0}$, and 
we have used the isomorphism (\ref{isom:2[C]}). 
The second integrations are derived from (\ref{chi:half}).  \\

The rest of this subsection is devoted to an explicit computation of $\DT_{4}(2[C])$.
\begin{lem}\label{lem:thick}
Let $\zZ \subset S^{d_i-d_0}(C) \times C$
be the universal divisor and
set
\begin{align*}
\aA=\oO_{S^{d_i-d_0}(C) \times C}(\zZ)
\boxtimes L_i^{-1}.
\end{align*}
We set $N_i=N-L_i \otimes t_i$ in the $T_0$-equivariant K-theory
of $C$, 
and denote by $p_S$, $\pi_S$ the projections from
$\Pic^{d_i}(C) \times S^{d_i-d_0}(C)$,
$S^{d_i-d_0}(C) \times C$
 to $S^{d_i-d_0}(C)$
respectively. Then we have the identity
\begin{align*}
&\int_{\Pic^{(d_0, d_i)}(C)}e\left(\dR \hH om_{\pi}(j_{\ast}\fF_0, j_{\ast}\fF_0)^{\rm{mov}}
+\dR \hH om_{\pi}(j_{\ast}\fF_0, j_{\ast}\fF_i \cdot t_i^{-1})^{\rm{mov}} \right) \\
&=-\lambda_1^{-2l_1-2}\lambda_2^{-2l_2-2}
(\lambda_1+\lambda_2)^{-2l_3-2} \\
&\cdot \int_{\Pic^{d_i}(C) \times S^{d_i-d_0}(C)}
p_S^{\ast}e \left((\dR \pi_{S\ast}\aA -\dR \pi_{S\ast}(\aA \boxtimes N_i)
+\dR \pi_{S\ast}(\aA \boxtimes \wedge^2 N) -
\dR \pi_{S\ast}(\aA \boxtimes \omega_C))t_i^{-1} \right).
\end{align*}
\end{lem}
\begin{proof}
The $T_0$-movable part of (\ref{chi:half}) is
\begin{align}\label{deg2:mov}
&\chi(F, F)^{1/2, \rm{mov}} \\
\notag
&=-\chi(N)+\chi(\wedge^2 N)
+(\chi(A)-\chi(A \otimes N_i)
+\chi(A \otimes \wedge^2 N)
-\chi(A \otimes \wedge^3 N))t_i^{-1}.
\end{align}
Suppose that $(F_0, F_i')$ corresponds
to $(F_i', Z)$ under the isomorphism
in Lemma~\ref{lem:Picd}.
Then
\begin{align*}
A=F_i \otimes F_0^{\vee}=F_i' \otimes F_0^{\vee} \otimes L_i^{-1}
=\oO_C(Z) \otimes L_i^{-1}.
\end{align*}
Therefore the desired identity holds.
\end{proof}
Similar to Proposition \ref{cor on DT4 C},
we have the following vanishing for higher genus.
\begin{cor}\label{cor on DT4 2C}
For $g(C)>0$, we have $\DT_4(2[C])=0$.
\end{cor}
For $g(C)=0$, we compute the integral in
Lemma~\ref{lem:thick} as follows.
\begin{lem}\label{lem:g=0}
Suppose that $g(C)=0$.
By setting $k=d_i-d_0+1$,
the integral 
\begin{align*}
\int_{\Pic^{d_i}(C) \times S^{d_i-d_0}(C)}
p_S^{\ast}e \left((\dR \pi_{S\ast}\aA -\dR \pi_{S\ast}(\aA \boxtimes N_i)
+\dR \pi_{S\ast}(\aA \boxtimes \wedge^2 N) -
\dR \pi_{S\ast}(\aA \boxtimes \omega_C))t_i^{-1} \right)
\end{align*}
in Lemma~\ref{lem:thick} for $i=1$ is
calculated as
\begin{align*}
&A(l_1, l_2, l_3, k) \cneq
\mathrm{Res}_{h=0}
\left\{h^{-k}(-\lambda_1+h)^2 (\lambda_2+h)^{k+l_2} \right. \\
&\left.(-\lambda_1-\lambda_2+h)^{k+l_3}
(-\lambda_1+\lambda_2+h)^{l_1-l_2-k}
(-2\lambda_1-\lambda_2+h)^{l_1-l_3-k}
(-2\lambda_1+h)^{k-2-2l_1} \right\}.
\end{align*}
The integral for $i=2$ is given by
\begin{align*}
&B(l_1, l_2, l_3, k) \cneq \mathrm{Res}_{h=0}
\left\{h^{-k}(-\lambda_2+h)^2 (\lambda_1+h)^{k+l_1} \right. \\
&\left. (-\lambda_2-\lambda_1+h)^{k+l_3}
(-\lambda_2+\lambda_1+h)^{l_2-l_1-k}
(-2\lambda_2-\lambda_1+h)^{l_2-l_3-k}
(-2\lambda_2+h)^{k-2-2l_2} \right\}.
\end{align*}
\end{lem}
\begin{proof}
For $g(C)=0$, we have $\Pic^{d_i}(C)=\Spec \mathbb{C}$
and $S^{d_i-d_0}(C)=\mathbb{P}^{d_i-d_0}$.
The universal divisor $\zZ \subset \mathbb{P}^{d_i-d_0} \times \mathbb{P}^1$
is a $(1, d_i-d_0)$-divisor, so we have
\begin{align*}
\aA=\oO_{\mathbb{P}^{d_i-d_0} \times \mathbb{P}^1}(1, d_i-d_0-l_i).
\end{align*}
Therefore
we have (here we write $\oO(1)=\oO_{\mathbb{P}^{d_i-d_0}}(1)$)
\begin{align*}
&\dR \pi_{S\ast}\aA=\oO(1)^{\oplus k-l_1}, \
\dR \pi_{S\ast}(\aA \boxtimes \omega_C)
=\oO(1)^{\oplus k-l_1-2}, \\
&\dR \pi_{S\ast}(\aA \boxtimes N_1)=\oO(1)^{\oplus k-l_1+l_2}t_2
\oplus \oO(1)^{\oplus k-l_1+l_3}t_3, \\
&\dR \pi_{S\ast}(\aA \boxtimes \wedge^2 N)
=\oO(1)^{\oplus k+l_2}t_1 t_2
\oplus \oO(1)^{\oplus k+l_3}t_1 t_3
\oplus \oO(1)^{\oplus k-2l_1-2}t_2 t_3.
\end{align*}
Let $h=c_1(\oO(1))$.
Then the desired integral is the $h^{k-1}$-part
of
\begin{align*}
&(-\lambda_1+h)^{2}
(-\lambda_1+\lambda_2+h)^{-k+l_1-l_2}(-\lambda_1+\lambda_3+h)^{-k+l_1-l_3} \\
&(\lambda_2+h)^{k+l_2}(\lambda_3+h)^{k+l_3}(-\lambda_1+\lambda_2+\lambda_3+h)^{k-2-2l_1}.
\end{align*}
Using $\lambda_1+\lambda_2+\lambda_3=0$, we obtain the desired result.
The case of $i=2$ is similar.
\end{proof}

\begin{cor}\label{DT4 for 2C}
For $g(C)=0$,
we have
\begin{align*}
\mathrm{DT}_{4}(2[C])=
&-\lambda_1^{-2l_1-2}\lambda_2^{-2l_2-2}(\lambda_1+\lambda_2)^{-2l_3-2} \\
& \cdot \left(
\sum_{\begin{subarray}{c}
1\leqslant k\leqslant l_1, \\
k \equiv l_1 \ (\mathrm{mod} 2)
\end{subarray}}
A(l_1, l_2, l_3, k)
+
\sum_{\begin{subarray}{c}
1\leqslant k\leqslant l_2, \\
k \equiv l_2 \ (\mathrm{mod} 2)
\end{subarray}}
B(l_1, l_2, l_3, k) \right).
\end{align*}
\end{cor}
\begin{proof}
Since
$l_1+l_2+l_3=-2$ and we have assumed
$l_1 \geqslant l_2 \geqslant l_3$, so
$l_1 \geqslant 0>l_3$.
In particular, by the inequality
\begin{align}\label{ineq:di}
0 \leqslant d_0 \leqslant \frac{l_i-1}{2}
\end{align}
in the definition
 of $\DT_4(2[C])$, there is
no contribution from $L_3$-thickened sheaves

By setting $k=d_i-d_0+1$,
since $d_0+d_i=l_i-1$, we have
\begin{align}\label{d:k}
d_0=\frac{l_i-k}{2}, \ d_i=\frac{l_i+k}{2}-1.
\end{align}
The inequality
(\ref{ineq:di})
is then equivalent to $1\leqslant k \leqslant l_i$.
Conversely given an integer $k$ satisfying
$1\leqslant k\leqslant l_i$, there exists
$(d_0, d_i) \in \mathbb{Z}^2$
satisfying (\ref{ineq:di}) and (\ref{d:k})
if and only if $k\equiv l_i(\mod 2)$.
Therefore the desired identity holds by
Lemma~\ref{lem:thick} and
 Lemma~\ref{lem:g=0}.
\end{proof}

\subsection{Equivariant version of Conjecture \ref{conj:GW/GV} for local curves}
We propose an equivariant version of Conjecture \ref{conj:GW/GV} (without insertions) for degree two case (the degree one case is given by Corollary \ref{cor on deg one local curve}).
\begin{conj}\label{conj on local P1}
For any smooth projective curve $C$ and line bundles $L_i$ ($i=1,2,3$) on $C$ with $L_1 \otimes L_2 \otimes L_3 \cong \omega_C$, we have 
the identity
\begin{align*} \mathrm{GW}_{0,2[C]}=\mathrm{DT}_{4}(2[C])+\frac{1}{8}\mathrm{DT}_{4}([C])\in\mathbb{Q}(\lambda_1, \lambda_2) \end{align*}
after substituting $\lambda_3=-\lambda_1-\lambda_2$.
\end{conj}
\begin{rmk}
As there is no insertion here, the coefficient of $\mathrm{DT}_{4}([C])$ is $1/2^3$ instead of $1/2^2$ 
in Conjecture \ref{conj:GW/GV}.
\end{rmk}
When $g(C)\geqslant1$, the above conjecture is obviously true by Proposition \ref{cor on DT4 C} and Corollary \ref{cor on DT4 2C}. 
For $g(C)=0$ case, we fix constants $l_i\in \mathbb{Z}$ ($i=1,2,3$) with $l_1+l_2+l_3=-2$ and may assume $l_1 \geqslant l_2 \geqslant l_3$ without loss of generality, the above conjecture is expressed in terms of a polynomial relation in variables $\{\lambda_i^{\pm}\}$, we verify the conjecture with the help of software 'Mathematica' in examples as follows.
\begin{thm}\label{prop on local P1}
When $g(C)=0$ and denote $\deg(L_i)=l_i$, then Conjecture \ref{conj on local P1} is true if
\begin{equation}|l_1|\leqslant10, \quad |l_2|\leqslant10.  \nonumber \end{equation}
\end{thm}
In the appendix, we will list a computational example (the case of $l_1=8$, $l_2=6$, $l_3=-16$) given by 'Mathematica' and one
can see there are really non-trivial cancellations involved to make Conjecture \ref{conj on local P1} true in this case. \\

For curve classes of higher degree, we may define $\mathrm{DT}_{4}(d[C])\in\mathbb{Q}(\lambda_1, \lambda_2)$ and have
\begin{conj}
For any smooth projective curve $C$, line bundles $L_i$ ($i=1,2,3$) on $C$ with $L_1 \otimes L_2 \otimes L_3 \cong \omega_C$
and $d\geqslant1$, we have the identity
\begin{align*} \mathrm{GW}_{0, d[C]}=\sum_{k|d}\frac{1}{k^{3}}\cdot\mathrm{DT}_{4}\Big(\frac{d}{k}[C]\Big)\in\mathbb{Q}(\lambda_1, \lambda_2) \end{align*}
after substituting $\lambda_3=-\lambda_1-\lambda_2$.
\end{conj}

\appendix
\addcontentsline{toc}{section}{Appendices}


\section{Genus zero GV/DT conjecture for CY 3-folds}\label{append:CY3}
Here we recall the genus zero GV/DT conjecture for CY 3-folds and 
explain how it can be derived from the conjectural GW/PT correspondence, a geometric 
vanishing conjecture and wall-crossing in the derived category. 

Let $Y$ be a smooth projective CY 3-fold. 
For $\beta \in H_2(Y, \mathbb{Z})$, 
the virtual dimension of $\overline{M}_g(Y, \beta)$ is 
zero due to the CY3 condition of $Y$. 
Hence we have the GW invariant without insertion
\begin{align*}
\mathrm{GW}_{g, \beta} =\int_{[\overline{M}_g(Y, \beta)]^{\rm{vir}}}1 \in \mathbb{Q}. 
\end{align*}
Its generating series is uniquely written as 
\begin{align}\label{GW/GV:Y}
\sum_{g\geqslant 0, \beta>0}
\mathrm{GW}_{g, \beta}\lambda^{2g-2} t^{\beta}=
\sum_{\beta>0}
\sum_{g\geqslant 0, k\geqslant 1}
\frac{n_{g, \beta}^{\rm{GW}}}{k}
(-1)^{g-1}
((-q)^{\frac{k}{2}}-(-q)^{-\frac{k}{2}})^{2g-2} t^{k\beta}
\end{align}
for some $n_{g, \beta}^{\rm{GW}} \in \mathbb{Q}$. 
Here $q=-e^{i\lambda}$. 
The GV conjecture \cite{GV} claims that $n_{g, \beta}^{\rm{GW}}$ are integers. 
If we focus on the genus zero invariants, the formula (\ref{GW/GV:Y})
implies
\begin{align*}
\mathrm{GW}_{0, \beta}=\sum_{k\geqslant 1, k|\beta}\frac{1}{k^3} n_{0, \beta/k}^{\rm{GW}}. 
\end{align*}
On the other hand, let $M_{\beta}$ be the moduli space of 
one dimensional stable sheaves $E$ on $Y$ with $[E]=\beta$, 
$\chi(E)=1$. 
Then there is a zero-dimensional virtual fundamental cycle 
on $M_{\beta}$, and its integral
yields the $\mathrm{DT}_3$ invariant
\begin{align*}
\mathrm{DT}_3(\beta) \cneq \int_{[M_{\beta}]^{\rm{vir}}} 1 \in \mathbb{Z}. 
\end{align*}
The genus zero GV/DT conjecture for CY 3-folds is as follows: 
\begin{conj}\emph{(\cite{Katz})}\label{conj:katz}
We have the identity 
$n_{0, \beta}^{\rm{GW}}=\mathrm{DT}_3(\beta)$. 
\end{conj}
Let $P_n(Y, \beta)$ be the moduli space of stable pairs $(F, s)$ on $Y$
such that $[F]=\beta$, $\chi(F)=n$. 
Similarly, we have the PT invariant \cite{PT}
\begin{align*}
P_{n, \beta} \cneq \int_{[P_n(Y, \beta)]^{\rm{vir}}} 1 \in \mathbb{Z}. 
\end{align*}
The logarithm of its generating series is uniquely written as 
\begin{align}\label{PT/GV:Y}
\log \left(
1+
\sum_{\beta>0, 
 n\in \mathbb{Z}}
P_{n, \beta}q^n t^{\beta} \right) 
=
\sum_{\beta>0}
\sum_{g\in \mathbb{Z}, k\geqslant 1}
\frac{n_{g, \beta}^{P}}{k}
(-1)^{g-1}
((-q)^{\frac{k}{2}}-(-q)^{-\frac{k}{2}})^{2g-2} t^{k\beta}. 
\end{align}
for some $n_{g, \beta}^{P} \in \mathbb{Z}$
with $n_{g, \beta}^P=0$ for $g\gg 0$. 
We have the following strong form of the 
GW/PT correspondence: 
\begin{conj}\label{conj:smnop}
We have the identity $n_{g, \beta}^{\rm{GW}}=n_{g, \beta}^P$
for any $g \in \mathbb{Z}$ and $\beta>0$. 
\end{conj}
We have the following: 
\begin{lem}\label{lem:conj:katz}
Suppose that $Y$ satisfies 
Conjecture~\ref{conj:smnop}. 
Then $Y$ satisfies Conjecture~\ref{conj:katz}. 
\end{lem}
\begin{proof}
Conjecture~\ref{conj:smnop} implies that 
$n_{g, \beta}^P=0$ for $g<0$. 
By~\cite[Theorem~6.4]{Toda2}, the
wall-crossing argument in the derived category shows 
the identity $n_{0, \beta}^P=\mathrm{DT}_3(\beta)$. 
Therefore $n_{0, \beta}^{\rm{GW}}=\mathrm{DT}_3(\beta)$ holds. 
\end{proof}
Note that the original GW/DT conjecture~\cite{MNOP}
together with DT/PT correspondence~\cite{Toda}
only claims the identity of (\ref{GW/GV:Y}) and (\ref{PT/GV:Y})
as rational functions of $q$. 
In order to further have the identity 
$n_{g, \beta}^{\rm{GW}}=n_{g, \beta}^P$, 
we need to know either one of the following properties: 
\begin{enumerate}
\item For any fixed $\beta$, we have $n_{g, \beta}^{\rm{GW}}=0$
for $g\gg 0$. 
\item For any $\beta$, we have $n_{g, \beta}^P=0$ for $g<0$. 
\end{enumerate}
Indeed if one the above conditions is satisfied, then 
the uniqueness of the expressions in the 
form of (\ref{GW/GV:Y}) or (\ref{PT/GV:Y}) shows the 
identity $n_{g, \beta}^{\rm{GW}}=n_{g, \beta}^P$. 
So assuming that the GW/PT conjecture holds, 
the conditions (1), (2), and Conjecture~\ref{conj:smnop} are equivalent. 

The vanishing $n_{g, \beta}^P=0$ for $g<0$ is 
nothing but Pandharipande-Thomas' strong rationality conjecture~\cite{PT}.
By the wall-crossing argument in the derived category~\cite{Toda2}, 
this vanishing is equivalent to the multiple cover formula
of generalized DT invariants~\cite{JS}. 
Let 
$N_{n, \beta} \in \mathbb{Q}$ be the generalized DT invariant
which counts one dimensional semistable sheaves $E$ on $Y$
with $[E]=\beta$, $\chi(E)=n$. 
Its multiple cover conjecture is given as follows:
\begin{conj}\emph{\cite{JS, Toda2}}
We have the identity
\begin{align}\label{N:mult}
N_{n, \beta}=\sum_{k\geqslant 1, k|(n, \beta)}
\frac{1}{k^2} \DT_3(\beta/k).
\end{align}
\end{conj}
We have the following lemma: 
\begin{lem}\label{lem:primitive}
For a CY 3-fold $Y$, suppose that the GW/PT conjecture~\cite{MNOP} holds. 
Moreover suppose that for a fixed $\beta \in H_2(Y, \mathbb{Z})$, 
the identity (\ref{N:mult}) holds. 
Then for any $k\geqslant 1$
with $k|\beta$, we have the identity
$n_{0, \beta/k}^{\rm{GW}}=\DT_3(\beta/k)$. 
In particular, $n_{0, \beta}=\DT_3(\beta)$ holds 
for any primitive curve class $\beta$. 
\end{lem}
\begin{proof}
By the argument of~\cite[Theorem~6.4]{Toda2}, 
the identity (\ref{N:mult}) implies 
that $n_{g, \beta/k}^P=0$ for any $g<0$ and $k\geqslant 1$
with $k|\beta$, and $n_{0, \beta/k}^{P}=\DT_3(\beta/k)$ holds. 
Then comparing (\ref{GW/GV:Y}) with (\ref{PT/GV:Y}), 
we obtain 
$n_{0, \beta/k}^{\rm{GW}}=\DT_3(\beta/k)$. 
The identity (\ref{N:mult}) always holds
 for primitive curve class $\beta$ by~\cite[Lemma~2.12]{Toda3}. 
Therefore $n_{0, \beta}=\DT_3(\beta)$
holds for any primitive $\beta$. 
\end{proof}
Then using the result of~\cite{PP}, we have the following: 
\begin{cor}\label{cor:CI}
Let $Y$ be a complete intersection CY 3-fold in the 
product of projective 
spaces $\mathbb{P}^{n_1} \times \cdots \times \mathbb{P}^{n_k}$. 
Then for any primitive curve class $\beta$ on $Y$, 
we have $n_{0, \beta}^{\rm{GW}}=\DT_3(\beta)$. 
\end{cor}
\begin{proof}
The GW/PT conjecture is proved for such CY threefolds in~\cite{PP}.
Then the result follows from Lemma~\ref{lem:primitive}. 
\end{proof}
In the case of some toric CY 3-fold, we can 
derive Conjecture~\ref{conj:katz}
from the vanishing $n_{g, \beta}^{\rm{GW}}=0$ for 
$g\gg 0$. 
Let $S$ be a smooth toric 
del-Pezzo surface and 
consider the non-compact CY 3-fold 
$Y=\mathrm{Tot}_S(K_S)$. 
Although $Y$ is non-compact, 
all the relevant objects 
(stable maps, stable pairs, stable sheaves)
are supported on the zero section of $Y \to S$. 
Therefore all the above 
invariants, results are obtained as in the 
projective CY 3-fold case. 
We have the following corollary: 
\begin{cor}\label{cor:toric}
Let $S$ be a smooth toric
del-Pezzo surface. 
Then Conjecture~\ref{conj:katz} holds for 
$Y=\mathrm{Tot}_S(K_S)$. 
\end{cor}
\begin{proof}
Since $S$ is toric, the 
Conjecture~\ref{conj:smnop}
holds by~\cite{MNOP, Konishi}. 
Therefore the result follows by Lemma~\ref{lem:conj:katz}. 
\end{proof}
\begin{rmk}
In particular for $Y=\mathrm{Tot}_S(K_S)$ 
in Corollary~\ref{cor:toric},
the identity (\ref{N:mult}) holds.
We don't know how to prove this without 
using the GW/PT correspondence~\cite{MNOP}.   
\end{rmk}

\section{An orientability result for smooth moduli spaces of one dimensional stable sheaves}
Let $X$ be a smooth projective CY 4-fold and $M_\beta$ be the moduli space of one dimensional stable sheaves of Chern character $(0,0,0,\beta,1)$.
We denote $\mathcal{L}$ to be its determinant line bundle (see (\ref{det line bdl})) and $Q$ the non-degenerate quadratic form induced by Serre duality. Then we have the following:
\begin{prop}
If $M_\beta$ is a normal variety, then $(\mathcal{L},Q)$ has an orientation.
\end{prop} 
\begin{proof}
In fact, we are left to show $\mathcal{L}\cong \mathcal{O}_{M_\beta}$.
Since if $\mathcal{L}\cong \mathcal{O}_{M_\beta}$, the square $\mathcal{L}^{2}\cong \mathcal{O}_{M_\beta}$ of this isomorphism, although may be different from the one given by Serre duality (\ref{Serre duali}), its difference with that one gives an isomorphism $\mathcal{O}_{M_\beta}\cong\mathcal{O}_{M_\beta}$, which has a square root as $M_\beta$ is compact.

Then the argument follows from \cite[Proposition 3.13]{MT} which we adapt to our case as follows.
Note that for a normal variety, a holomorphic line bundle is determined by its restriction to the smooth locus. Without loss of generality, we may
assume $M_\beta$ is a smooth variety. Let $\mathcal{E} \in \Coh(X \times M_\beta)$ be the universal family,
then $[\mathcal{E}] \in K(X \times M_\beta)$ lies in the subgroup $K^{\geqslant 3}(X \times M_\beta)$ generated
by sheaves with codimension $\geqslant3$. A classical result of Grothendieck gives 
\begin{equation}[\dR \hH om(\mathcal{E}, \mathcal{E})]=[\mathcal{E}^{\vee}] \otimes [\mathcal{E}] \in K^{\geqslant 6}(X \times M_\beta). \nonumber \end{equation} 
As $X$ is a 4-fold, we have 
\begin{equation}[\dR\pi_{M_\beta,*}\big(\dR \hH om(\mathcal{E},\mathcal{E})\big)] \in K^{\geqslant 2}(M_\beta), \nonumber \end{equation} 
where $\pi_{M_\beta} : X \times M_\beta \to M_\beta$ is the projection. By taking determinant, we are done. 
\end{proof}

\section{Software code and explicit computations of an example}
In the following, we list the computational result of an example (the case of $l_1=8$, $l_2=6$, $l_3=-16$) for Conjecture \ref{conj on local P1} and
Theorem \ref{prop on local P1} with the help of software 'Mathematica'.  \\

$\textbf{Software code of Mathematica}$.  \\

$\mathrm{DT_4([C])}=\lambda_1^{-1-l_1}\lambda_2^{-1-l_2}(-\lambda_1-\lambda_2)^{-1-l_3}$,  \\

$A=\mathrm{Sum}[\mathrm{Mod}[(k-l_1+1) ,2](-\lambda_1^{-2l_1-2})(\lambda_2^{-2l_2-2})((\lambda_1+\lambda_2)^{-2l_3-2})\cdot \\
{}\quad\quad \quad \mathrm{Residue}[(h^{-k})(-\lambda_1+h)^{2}(\lambda_2+h)^{k+l_2}(-\lambda_1-\lambda_2+h)^{k+l_3}(-\lambda_1+\lambda_2+h)^{l_1-l_2-k}\cdot \\
{}\quad\quad \quad(-2\lambda_1-\lambda_2+h)^{l_1-l_3-k}(-2\lambda_1+h)^{k-2-2l_1},\{h,0\}], \{k,l_1\}]$,  \\

$B=\mathrm{Sum}[\mathrm{Mod}[(k-l_2+1) ,2](-\lambda_1^{-2l_1-2})(\lambda_2^{-2l_2-2})((\lambda_1+\lambda_2)^{-2l_3-2})\cdot \\
{}\quad\quad \quad \mathrm{Residue}[(h^{-k})(-\lambda_2+h)^{2}(\lambda_1+h)^{k+l_1}(-\lambda_1-\lambda_2+h)^{k+l_3}(-\lambda_2+\lambda_1+h)^{l_2-l_1-k}\cdot \\
{}\quad\quad \quad(-2\lambda_2-\lambda_1+h)^{l_2-l_3-k}(-2\lambda_2+h)^{k-2-2l_2},\{h,0\}], \{k,l_2\}], $ \\

$\mathrm{DT_4(2[C])}=A+B$,  \\

$\mathrm{GW_{0,2[C]}}=-\frac{1}{8}(\lambda_1^{-2l_1-1})(\lambda_2^{-2l_2-1})((\lambda_1+\lambda_2)^{-2l_3-3})\cdot \\
{}\quad \Big(\left(\mathrm{Sum}[(-1)^{i-1}(l_1-(i-1))^{2},\{i,l_1+1\}]\right)\lambda_1^{-2}(\lambda_1+\lambda_2)^{2}+ \\
{}\quad  \left(\mathrm{Sum}[(-1)^{i-1}(l_2-(i-1))^{2},\{i,l_2+1\}]\right)\lambda_2^{-2}(\lambda_1+\lambda_2)^{2}+ \\
{}\quad \left(\mathrm{Sum}[(-1)^{i-1}(-l_3-i)^{2},\{i,-l_3\}]\right)+l_1l_2\lambda_1^{-1}\lambda_2^{-1}(\lambda_1+\lambda_2)^{2}-l_2l_3\lambda_2^{-1}(\lambda_1+\lambda_2)-l_1l_3\lambda_1^{-1}(\lambda_1+\lambda_2) \Big)$.  \\

$\textbf{An explicit example}$. For $l_1=8$, $l_2=6$, $l_3=-16$, we have  \\

$\mathrm{DT_4([C])}=\frac{(-\lambda_1-\lambda_2)^{15}}{\lambda_1^{9}\lambda_2^{7}}$,  \\

$\mathrm{DT_4(2[C])}=-\frac{(\lambda_1+\lambda_2)^{15}}{8\lambda_1^{19}\lambda_2^{15}}
\left(21\lambda_1^{18}+480\lambda_1^{17}\lambda_2+5012\lambda_1^{16}\lambda_2^{2}+31776\lambda_1^{15}\lambda_2^{3}+
137460\lambda_1^{14}\lambda_2^{4}+432208\lambda_1^{13}\lambda_2^{5}+ \right.$ \\
${} \quad\quad \quad 1026480\lambda_1^{12}\lambda_2^{6}+1886976\lambda_1^{11}\lambda_2^{7}+2726437\lambda_1^{10}\lambda_2^{8}
+3123120\lambda_1^{9}\lambda_2^{9}+2845128\lambda_1^{8}\lambda_2^{10}+2057120\lambda_1^{7}\lambda_2^{11}+$ \\
${}\left. \quad\quad \quad 1171716\lambda_1^{6}\lambda_2^{12}+518448\lambda_1^{5}\lambda_2^{13}+174160\lambda_1^{4}\lambda_2^{14}+
42816\lambda_1^{3}\lambda_2^{15}+7245\lambda_1^{2}\lambda_2^{16}+752\lambda_1\lambda_2^{17}+36\lambda_2^{18}\right)$,  \\

$\mathrm{GW_{0,2[C]}}=-\frac{(\lambda_1+\lambda_2)^{29}}
{8\lambda_1^{19}\lambda_2^{15}}(21\lambda_1^{4}+186\lambda_1^{3}\lambda_2+497\lambda_1^{2}\lambda_2^{2}+248\lambda_1\lambda_2^{3}+36\lambda_2^{4})$, \\

$\frac{1}{8}\mathrm{DT_4([C])}+\mathrm{DT_4(2[C])}-\mathrm{GW_{0,2[C]}}=0$.


\end{document}